\documentclass[a4paper, 11pt]{amsart}

\usepackage{amsmath}
\usepackage{amsthm}
\usepackage{comment}

\usepackage[utf8]{inputenc}

\usepackage{xcolor}
\definecolor{green}{RGB}{0,127,0}
\definecolor{red}{RGB}{105,89,205}
\usepackage[colorlinks=true]{hyperref}

\theoremstyle{plain}
\newtheorem{theorem}{Theorem}[section]
\newtheorem*{thm}{Theorem}

\newtheorem{fact}[theorem]{Fact}
\newtheorem{lemma}[theorem]{Lemma}
\newtheorem{corollary}[theorem]{Corollary}
\newtheorem*{corr}{Corollary}
\newtheorem{proposition}[theorem]{Proposition}
\newtheorem{conjecture}[theorem]{Conjecture}
\newtheorem*{conj}{Conjecture}
\newtheorem{ass}[theorem]{Assumption}

\theoremstyle{definition}
\newtheorem{definition}[theorem]{Definition}
\newtheorem{example}[theorem]{Example}
\newtheorem*{ex}{Example}
\newtheorem{question}[theorem]{Question}

\newtheorem{remark}[theorem]{Remark}

\theoremstyle{remark}
\newtheorem{claim}[theorem]{Claim}

\numberwithin{figure}{section}
\numberwithin{equation}{section}

\DeclareMathOperator{\alt}{Alt}
\DeclareMathOperator{\supp}{supp}

\DeclareMathOperator{\GL}{GL}

\DeclareMathOperator{\SL}{SL}
\DeclareMathOperator{\PSL}{PSL}

\DeclareMathOperator{\rk}{rk}

\DeclareMathOperator{\lcm}{lcm}

\newcommand{\IET}{\mathcal{IET}}

\newcommand{\CC}{{\mathbb C}}

\newcommand{\PPsl}{{\mathcal{PSL}}}
\newcommand{\Pp}{{\mathbb P}}
\newcommand{\PS}{{\mathcal{SL}}}
\newcommand{\N}{{\mathbb N}}
\newcommand{\Z}{{\mathbb Z}}

\newcommand{\R}{{\mathbb R}}

\newcommand{\F}{{\mathbb F}}

\newcommand{\A}{{\mathcal A}}
\newcommand{\E}{{\mathcal E}}

\newcommand{\G}{{\mathcal G}}
\newcommand{\X}{{\mathcal X}}
\newcommand{\B}{{\mathcal B}}

\newcommand{\U}{{\mathcal U}}

\renewcommand{\S}{{\mathcal S}}

\renewcommand{\E}{{\mathcal E}}

\renewcommand{\H}{{\mathcal H}}

\renewcommand{\ll}{{\ell}}

\begin{document}

\title{New compactness theorem for metric ultraproducts and simplicity}

\author{Jakub Gismatullin}

\address{Instytut Matematyczny Uniwersytetu Wroc{\l}awskiego, pl. Grunwaldzki 2/4, 50-384 Wroc{\l}aw, Poland \& Instytut Matematyczny Polskiej Akademii Nauk, ul. {\'S}niadeckich 8, 00-656 Warszawa, Poland}

\email{jakub.gismatullin@uwr.edu.pl}

\thanks{\noindent The first author is supported by the National Science Centre, Poland NCN grants no.  2014/13/D/ST1/03491 and 2017/27/B/ST1/01467.}

\author{Krzysztof Majcher}

\address{Department of Computer Science, Faculty of Fundamental Problems of Technology, Wroc{\l}aw University of Science and Technology, Wybrze{\.z}e Wyspianskiego 27, 50-370 Wroc{\l}aw, Poland}

\email{k.majcher@pwr.wroc.pl}

\author{Martin Ziegler}

\address{Mathematisches Institut, Albert-Ludwigs-Universit{\"a}t Freiburg, D-79104 Freiburg, Germany}

\email{ziegler@uni-freiburg.de}

\date{\today}


\begin{abstract}
We give a new compactness theorem for any metric ultraproducts of family of metric groups.

As an application we characterize simplicity of metric ultraproducts of groups and give a couple of examples of new simple groups which are metric ultraproducts of finite and infinite symmetric groups, linear groups and interval exchange transformations group.
\end{abstract}

\maketitle

\tableofcontents


\section*{Introduction}

The ultraproduct construction is playing an important role in model theory, topology and algebra. A more general construction is that of \emph{metric ultraproduct}, when the objects are equipped with some king of invariant metrics. This paper is about metric ultraproducts of groups, that is groups equipped with conjugacy invariant norms. The importance of metric ultraproducts to group theory became apparent recently, especially in the case of sofic groups \cite{pestov1}. Metric ultraproducts are currently intensively studied. The main open problem is the following conjecture.

\begin{conj} Every group is sofic, that is
every finitely generated group can be homomorphically embedded into a \emph{metric} ultraproduct $\S^*_{\text{met}}$ of permutation groups $\S=(S_n,\|\cdot\|_H)$, equipped with the normalised Hamming norm $\|\sigma\|_H=\frac{1}{n}\|\supp(\sigma)\|$.
\end{conj}

A \emph{non-sofic group} is a group that cannot be homomorphically embedded into $\S^*_{\text{met}}$. No example of a non-sofic group is known. However, some well known groups (e.g. Higman group) are not known to be sofic. We believe that our new compactness theorem \ref{thm:sim_b} better explain the structure of metric ultraproducts and give a new method to study them.

Let us briefly explain the plan. Our first aim is to prove general compactness theorem \ref{thm:sim2} valid for all metric ultraproducts, also for families of metric groups with \emph{unbounded norms}. That is why we work with \emph{finitary element-subgroup} $\G^*_{\text{met, fin}}$ (Definition \ref{def:finn}), which is a normal subgroup of metric ultraproduct $\G^*_{\text{met, fin}}$. 
We give several applications of Theorems \ref{thm:sim_b}, \ref{thm:sim2}. We characterize when $\G^*_{\text{met, fin}}$ is a simple group (Theorem \ref{thm:SimpFam}, Section \ref{sec:sim}). Our condition for simplicity involves covering of arbitrary big balls by neighbourhoods of conjugacy classes. We also study bounded and uniform simplicity in Section \ref{sec:ub}, and topological simplicity in Section \ref{sec:top}. We provide a couple of examples, where our characterizations give new examples of simple groups.
We also study bounded generation and perfectness of $\G^*_{\text{met}}$ in Sections \ref{sec:bg}, \ref{sec:perf} and metric version of being torsion group in Section \ref{sec:tors}. Section \ref{sec:tree} is about set theoretic consideration on some well-founded trees related with metric internal coverings (Theorem \ref{thm:tree}).

Let us give more details and explain the notion of metric ultraproduct. Suppose $\G = (G_n,\|\cdot\|_n)_{n\in \N}$ is a family of metric groups, so each $G_n$ is equipped with a conjugacy invariant norm $\|\cdot\|_n$ (see Definition \ref{def:met}). A \emph{metric ultraproduct} of $\G$ is denoted by $\G^*_{\text{met}}$ (Definition \ref{def:inff}).  $\G^*_{\text{met}}$ is again equipped with a norm and metric:
\[\|\cdot\|\colon \G^*_{\text{met}}\to [0,\infty]\text{ defined as }\left\|(g_n)/\E\right\| = \lim_{n\to \U} \left\|g_n\right\|_n.\]

Elements of a finite norm in $\G^*_{\text{met}}$ form a normal subgroup of $\G^*_{\text{met}}$, which we denote by $\G^*_{\text{met, fin}}$. Sometimes we work under the following assumption: \[\sup_{n\in\N}\|\cdot\|_n <\infty,\] which we name as \emph{bounded case}. Under this assumption $\G^*_{\text{met,fin}}=\G^*_{\text{met}}$. The core notion we are going to use is that of \emph{metrically internal subset} of  $\G^*_{\text{met}}$: $X\subseteq \G^*_{\text{met}}$ is \emph{metrically internal} if there is a collection of sets $\{X_n\}_{n\in\N}$, $X_n\subseteq G_n$ such that
\[X = \frac{X_0\times X_1\times X_2\times\cdots}{\E},\] where $\E$ is the infinitesimal subgroup (\ref{inf}). We state below a new compactness theorem for $\G^*_{\text{met}}$ in the bounded case. The general version is in Theorem \ref{thm:sim2} and Remark \ref{rem:gen}. By $\B(\varepsilon)$ in a metric group $(G,\|\cdot\|)$ we mean a ball around $e$ of radius $\varepsilon$.

\begin{thm}\ref{thm:sim_b} (bounded case)
The following conditions are equivalent.
\begin{enumerate}
\item $\G^*_{\text{met}} = \bigcup_{m\in\N} X_m$
\item For any countable infinite sequence of positive reals $(\varepsilon_0,\varepsilon_1,\ldots)\subset\R_{>0}$ there is $N\in\N$ such that
\begin{equation*}
\G^*_{\text{met}} =  X_{0} \B(\varepsilon_0) \cup \ldots \cup X_{N} \B(\varepsilon_N),\text{ holds in }\G^*_{\text{met}}.
\end{equation*}
Which is equivalent with the clause: for $\U$-almost all $n\in\N$
\begin{equation*}
G_n =  X_{0,n} \B(\varepsilon_0) \cup \ldots \cup X_{N,n} \B(\varepsilon_N),\text{ holds in }G_n.
\end{equation*}
\end{enumerate}
\end{thm}

Combining this result with Baire category theorem we achieve the following result.

\begin{thm} \ref{thm:ccomp} (bounded case)
Let $(X_n)_{n \in \N}$ be a increasing sequence of internal subsets of $\G^*_{\text{met}}$.
If $\G^*_{\text{met}} = \bigcup_{m\in\N} X_m$, then there is $N\in\N$ such that
\begin{equation*}
\G^*_{\text{met}} =  X_N^2.
\end{equation*}
\end{thm}

As corollary we get result below on  perfectness. 

\begin{corr} \ref{cor:perf}  (bounded case)
Suppose $\G^*_{\text{met}}$ is perfect, then $\G^*_{\text{met}}$ is uniformly perfect, i.e. there is $N\in\N$ such that every element of $\G^*_{\text{met}}$ is a product of $N$ commutators.
\end{corr}

We study also torsion and almost-torsion elements in $\G^*_{\text{met}}$. Let us provide our notion of $\varepsilon$-torsion element (Definition \ref{def:etors}). Fix $\varepsilon > 0$ and metric group $(G,\|\cdot\|)$.
\begin{itemize}
    \item $G$ is called \emph{$\varepsilon$-torsion}, if for every $g\in G$ there is $N\in\N$ such that $\|g^N\|< \varepsilon$.
    \item $G$ is called \emph{almost uniformly $\varepsilon$-torsion} if there is $N\in\N$, such that for every $g\in G$ there is $m\le N$ with $\|g^m\|< \varepsilon$.
\end{itemize}

Our compactness theorem gives. 

\begin{corr}\ref{cor:tors} (bounded case) If $\G^*_{\text{met}}$ is $\varepsilon$-torsion group, then $G$ and $\G^*_{\text{met}}$ is almost uniformly $2\varepsilon$-torsion.
\end{corr}

As mentioned above, a large part of article is devoted to simplicity. For standard (discrete) ultraproduct simplicity is equivalent with uniform simplicity (Definition \ref{def:bsim}). In metric case we get \emph{bounded simplicity}.

\begin{corr} \ref{cor:bbsim} (bounded case)
If $\G^*_{\text{met}}$ is simple, then $\G^*_{\text{met}}$ is boundedly simple, that is for any element $g \in G$ there is $N\in\N$ such that $C_{N}(g,G):= \left(g^{G} \cup g^{-1 G}\right)^{\leq N}=G$.
\end{corr}

Every simple metric ultraproduct known to us actually has a stronger property, which we call \emph{metric uniform simplicity} (Definition \ref{msf}):

\begin{quote} A metric group $(G,\|\cdot\|)$ is metrically uniformly simple, if fo all $r>0$, there is $N \in \N$ such that $C_{N}(g,G)=G$ holds for all $g \in G$ with $\|g\|>r$. 
\end{quote}
We conjecture that.

\begin{conj}\ref{con:gw}
If a metric ultraproduct $(\G_{\text{met}},\|\cdot\|)$ is simple, then $\G_{\text{met}}$ must be metrically uniformly simple.
\end{conj}

Section \ref{sec:met} is devoted to the study of Conjecture \ref{con:gw}. We introduce there an assumption, called $(\star)$-property in Definition \ref{def:star}, which allows to define an analogy of subgroup of infinitesimal sequences in (possibly non-metric) product of groups (Definition \ref{def:zz}, Fact \ref{ll}). We prove Conjecture \ref{con:gw} under $(\star)$-property in Theorem \ref{2Us->mus}. We also prove certain transfer result under the bounded assumption (Theorem \ref{Us->mus}, Fact \ref{**prop}): if we have two families of metric groups $\G=(G_n, \|\cdot\|_n)_{n \in \N}$ and $\G'=(G_n, \|\cdot\|'_n)_{n \in \N}$ with the same underlying groups $G_n$, such that $\G$ is metrically uniformly simple and $\G^{'*}_{\text{met}}$ is simple, then $\G^{'*}_{\text{met}}$ is  metrically uniformly simple too.

Large part of our work is devoted to construction a new examples of simple groups. In section \ref{sec:sym} we construct a family of simple groups based on permutation groups, as explained below.

\begin{ex}\ref{ex:per}\ \ 
For any sequence $\bar{c}=\left(c_n\right)_{n\in\N}$ of positive real numbers such that $\lim_{n\to\U}c_n=0$, consider \[\S(\bar{c})=\left(S_\infty, c_n\|\cdot \|_H\right)_{n \in \N},\]
    where $S_\infty=\bigcup_{n\in\N} S_n$. Then $\S(\bar{c})^*_{\text{met, fin}}$ is a simple group (by \ref{cor:stozek}). We do not know if 
    \[\S\left(\frac{1}{n}\right)^*_{\text{met, fin}} \cong \S\left(\frac{1}{n^2}\right)^*_{\text{met, fin}} ?\]
    We conjucture that each $\S(\bar{c})^*_{\text{met, fin}}$ is a universal sofic group \cite{pestov1}.
\end{ex}

Section \ref{sec:app} is devoted to Theorem \ref{thm:iet} which deals with approximation of a metric group by a simple family of metric groups. As an application of this theorem we prove a result in Section \ref{sec:iet} on $\IET$.

An \emph{interval exchange transformation} is a  bijective map $f\colon [0,1] \rightarrow [0,1]$ which is piecewise translation,  continuous on the right with finitely many discontinuity points.

The set of all interval exchange transformation with composition form a group, which we denote by $\IET$. A bi-invariant norm of an element $g \in \IET$ is a Lebesgue measure of its support:
\[\|g \|_\mu= \mu(\supp(g)).\]

\begin{thm}\ref{thm:iett}
Any metric ultrapower of $\IET$ with respect to $\|\cdot\|_\mu$  is a simple group, in fact metrically uniformly simple (Definition \ref{msf}).
\end{thm}

Using a model-theoretic argument and results of Liebeck-Shalev from \cite{liesha} we prove that metric ultraproduct of family of linear group is simple.

\begin{thm}\ref{PslC}
Let $G_n = (\SL_{m_n}(\CC),\ll_{J})$, for some $m_n\in\N_{>1}$.
Any metric ultraproduct of $(G_n)_{n\in\N}$ is a simple group. In fact, it is metrically  uniformly  simple (Definition \ref{msf}).
\end{thm}

In Section \ref{sec:dir_lim} we construct a simple group as a direct limit of system of linear groups.



\section{Metric ultraproducts} \label{sec:intro}

Let us explain some basic terms. Our definition of a metric group below follows \cite[Sec. 2.1]{nst} (see also \cite[Section 10.4.]{drutu}).

Let $[0,\infty]$ be $[0,\infty)\cup\{\infty\}$, where $\infty>r$ for all $r\in[0,\infty)$. We regard  $[0,\infty]$ as a compact space, where the neighbourhoods of $\infty$ are of the form $(c,\infty)$, for $c\in[0,\infty)$. 

\begin{definition}\label{def:met}
A \emph{bi-invariant metric} on a group $G$ is a metric $d\colon G\times G\to [0,\infty]$ such that $d(gx,gy)=d(x,y)=d(xg,yg)$ for every $g,x,y\in G$.
Every such metric comes from a \emph{bi-invariant} (also called  \emph{conjugacy invariant}) norm $\|\cdot\|\colon G\to [0,\infty]$ satisfying
\begin{enumerate}
    \item[(0)] $\|e\|=0$,
    \item $\|gh\|\leq \|g\| + \|h\|$,
    \item $\left\| g^{-1} \right\| = \|g\| = \|hgh^{-1}\|$,
    \item $\|g\|=0$ if and only if $g=e$,
\end{enumerate} 
for all $g,h\in G$.  That is, if $\|\cdot\|$  satisfies (0), (1), (2), and (3), then $d(x,y)=\left\|xy^{-1}\right\|$ is a bi-invariant metric on $G$. Moreover, if $d$ is a bi-invariant metric, then $\|g\| = d(g,e)$ is a bi-invariant norm. A \emph{pseudo bi-invariant norm} is a function on $G$, satisfying (0), (1) and (2), but not necessarily (3). Such a norm gives rise to a \emph{pseudo bi-invariant metric}. A \emph{metric group}, or a \emph{pseudo-metric group} $(G,\|\cdot\|)$ is a group $G$ equipped with some bi-invariant or pseudo bi-invariant norm $\|\cdot\|$.
\end{definition}

By $\U$ we always denote a non-principal ultrafilter on $\N$. Let us define standard (discrete) ultraproduct $\G^*$ of any family of groups $\G = (G_n,)_{n\in \N}$ with respect to $\U$:
\begin{equation}\label{def:ult}
\G^* = \frac{\prod_{n\in \N}G_n}{\E_d},\text{ where }\E_d = \left\{\left(g_n\right)_{n\in\N}\in \prod_{n\in \N}G_n :\ \{ n\in\N: g_n = e\}\in\U \right\}.
\end{equation}

Let us remind the notion of the limit with respect to ultrafilter. Suppose $\{s_n\}_{n\in\N}$ is a bounded sequence or reals. By $\lim_{n\to \U} s_n$ we mean the \emph{limit of $\{s_n\}_{n\in\N}$ over $\U$}, that is a real number $s\in\R$ which is uniquely defined by the following condition: for every $\varepsilon>0$ the set $\{n\in\N : |s_n-s|<\varepsilon\}$ belongs to $\U$. We are now ready to define metric ultraproduct.

\begin{definition}\label{def:inff}
Suppose $\G = (G_n,\|\cdot\|_n)_{n\in \N}$ is a family of pseudo-metric groups. A \emph{metric ultraproduct} $\G^*_{\text{met}}$ of $\G$ with regard to $\U$ is defined as a quotient group: 
\[\G^*_{\text{met}} = \frac{\prod_{n\in \N}G_n}{\E},\]
where $\E$ is the subgroup of \emph{infinitesimals}
\begin{equation}\label{inf}
\E = \left\{\left(g_n\right)_{n\in\N}\in \prod_{n\in \N}G_n : \lim_{n\to \U}\left\|g_n\right\|_n = 0 \right\}.
\end{equation}
\end{definition} 

\begin{remark}\label{rem:xxx}
\begin{enumerate}
    \item $\G^*_{\text{met}}$ is equipped with a pseudo bi-invariant norm
\begin{equation} \label{eq:norm}
\|\cdot\|\colon \G^*_{\text{met}}\to [0,\infty]\text{ defined as }\left\|(g_n)/\E\right\| = \lim_{n\to \U} \left\|g_n\right\|_n.
\end{equation}
    \item Since obviously $\E_d$ from (\ref{def:ult}) is a subgroup of $\E$ from (\ref{inf}), there is a natural epimorphism $\G^*\to \G^*_{\text{met}}$. 
\end{enumerate}
\end{remark}

We are interested mainly in elements of $\G^*_{\text{met}}$ of finite norm, that is why we define below $\G^*_{\text{met, fin}}$.

\begin{definition}\label{def:finn}
\emph{A finitary subgroup} $\G^*_{\text{met, fin}}$ of $\G^*_{\text{met}}$ is defined as
\begin{equation} \label{def:fin}
\G^*_{\text{met, fin}} = \left\{x\in \G^*_{\text{met}} : \|x\|<\infty\right\}.
\end{equation}
In fact $\G^*_{\text{met, fin}}=\G_{\text{fin}}/{\E}$, where $\G_{\text{fin}}= \left\{\left(g_n\right)_{n\in \N}\in \prod_{n\in \N}G_n : \sup_{n\in \N}\left\|g_n\right\|_n< \infty \right\}$. 
\end{definition}

Obviously, if the family $\{\|\cdot\|_n\}_{n\in\N}$ is uniformly bounded (i.e. $\sup_{n\in\N}\|\cdot\|_n <\infty$), then $\G^*_{\text{met, fin}}=\G^*_{\text{met}}$. However, in many interesting cases $\{\|\cdot\|_n\}_{n\in\N}$ is not uniformly bounded, so we work mainly with $\G^*_{\text{met, fin}}$.

When $\left(G_n,\|\cdot\|_n\right) = (G,\|\cdot\|)$, for all $n\in\N$, we use the symbol $G^*_{\text{met, fin}}$ for $\G^*_{\text{met, fin}}$. We have, in this case,  a standard isometric homomorphic embedding \[\varphi\colon G\to G^*_{\text{met, fin}}, \text{ given by }\varphi(g)=(g,g,g,\ldots)\E.\]

Let us record below some well-known properties of $\G^*_{\text{met, fin}}$.

\begin{lemma} \label{lem:propp}
Fix a family $\G$ of metric group. 
\begin{enumerate}
    \item The norm (\ref{eq:norm}) makes $\G^*_{\text{met, fin}}$ a topological group, that is, multiplication $\cdot\colon \G^*_{\text{met, fin}}\times \G^*_{\text{met, fin}}\to \G^*_{\text{met, fin}}$ is continuous with respect to $\|\cdot \|$.
    \item  $(\G^*_{\text{met, fin}}, \|\cdot \|)$ is a complete metric space, that is, every countable descending family of balls $\B = \{B_n : n\in\N\}$, $B_{n+1}\subseteq B_n\subseteq \G^*_{\text{met, fin}}$, has a non-empty intersection $\bigcap \B\neq \emptyset$.
\end{enumerate}
\end{lemma}

The first part of Lemma \ref{lem:propp} follows from the conjugacy-invariance of $\|\cdot\|$. The proof of the second part of Lemma \ref{lem:propp} is \cite[Corollary 10.64 (1)]{drutu}, which has roots in \cite[Proposition 4.2 (c), p. 364]{dries-wilkie}.

\subsection{Notation} Let us introduce some group-theoretic notation. For $g,h\in G$ we put $g^h=h^{-1}g h$ and $g^G=\left\{g^h : h\in G\right\}$. 
For a metric group $(G,\|\cdot\|)$, $g \in G$, a natural number $n\in\N$ and a positive real number $\varepsilon\in \R_{>0}$ define $\varepsilon$-balls around $g$: \begin{equation}\label{eq:not}
\B_\varepsilon(g,G)=g\cdot\B(\varepsilon,G) =\left\{h\in G : \left\|g^{-1}h\right\| < \varepsilon \right\},\ \B_{\le\varepsilon}(g,G)=g\cdot \B(\le \varepsilon,G) =\left\{h\in G : \left\|g^{-1}h\right\| \le \varepsilon \right\}.
\end{equation}

\section{New compactness theorem for metric ultraproducts of groups}\label{sec:compTh}

Let us fix throughout this section a family $\G = (G_n,\|\cdot\|_n)_{n\in \N}$ of metric groups. Note that $\|\cdot\|_n\colon G_n\to [0,\infty]$, so we allow $\|\cdot\|$ to have $\infty$ as a value. 

After introducing \emph{metrically internal sets} in Section \ref{sec:inter}, we give a new compactness theorem for $\G^*_{\text{met, fin}}$ in Section \ref{sec:comp}. Then we apply Baire category theorem to $(\G^*_{\text{met, fin}},\| \cdot\|)$ in Section \ref{sec:baire}, to get new results on metric ultraproducts.

\subsection{Metrically internal sets} \label{sec:inter}

Our new compactness theorem for $\G^*_{\text{met, fin}}$ involves \emph{metrically internal subsets} from Definition \ref{def:inter} below. This definition has roots in non-standard analysis (see e.g. \cite[Definition 10.34]{drutu}). It is more convenient to work with subsets of $\G^*_{\text{met}}$ rather that in $\G^*_{\text{met, fin}}$. That is why we define metrically internal subsets as some subsets of $\G^*_{\text{met}}$.

Let us briefly explain the situation for discrete metric, i.e. for standard ultraproduct $\G^*$. An \emph{internal subset $X$ of $\G^*$} is of the form  
\begin{equation}\label{inter}
X =  \frac{\prod_{n\in\N} X_n}{\E_d} = \frac{X_0\times X_1\times X_2\times\cdots}{\E_d},
\end{equation}
where each $X_n$ is a subset of $G_n$. This can be naturally generalized to 
$\G^*_{\text{met}}$ as in Definition \ref{def:inter} below.

\begin{definition} \label{def:inter}
A subset $X\subseteq \G^*_{\text{met}}$ is called \emph{metrically internal} if there is a collection of sets $\{X_n\}_{n\in\N}$, $X_n\subseteq G_n$ such that
\[X =  \frac{\prod_{n\in\N} X_n}{\E} = \frac{X_0\times X_1\times X_2\times\cdots}{\E},\] where $\E$ is the infinitesimal subgroup (see Definition \ref{def:inff}). 
\end{definition}

Let us give some examples and non-examples of metrically internal sets.

\begin{example}\label{ex:internal}
\begin{enumerate}
    \item \label{conj} A conjugacy class $\bar{g}^{\G^*_{\text{met}}}$, for any $\bar{g}=(g_n)_{n\in\N}/\E\in \G^*_{\text{met}}$, is metrically internal, as $\bar{g}^{\G^*_{\text{met}}} = \frac{\prod_{n\in\N} {g_n}^{G_n}}{\E}$.
    
    \item \label{ex:comm} The set of all commutators $\left\{\bar{g}^{-1}\bar{h}^{-1}\bar{g}\bar{h} : \bar{g},\bar{h}\in \G^*_{\text{met}}\right\}$ and the set of all $n$-powers $\left\{\bar{g}^{n} : \bar{g}\in \G^*_{\text{met}}\right\}$ are metrically internal (for the same reason as in (\ref{conj})).
    
    \item \label{balls} An open ball $\B_{\varepsilon}\left(\bar{g}\right)$
    and closed ball $\B_{\le\varepsilon}\left(\bar{g}\right)$
    for $\bar{g}=(g_n)_n/\E\in \G^*_{\text{met}}$ (see (\ref{eq:not})), may not be metrically internal, but there are canonical metrically internal sets in between:
    \[\B_{\varepsilon}\left(\bar{g},\G^*_{\text{met}}\right) \subseteq \frac{\prod_{n\in\N} \B_{\varepsilon}(g_n,G_n)}{\E} \subseteq \frac{\prod_{n\in\N} \B_{\le \varepsilon}(g_n,G_n)}{\E} \subseteq   \B_{\le\varepsilon}\left(\bar{g},\G^*_{\text{met}}\right).\]
    
    \item It is not true in general that a definable subset of $\G^*_{\text{met}}$ corresponds to definable subsets from coordinates. That is, fix $\varphi$  a formula in some first order logic, then the following natural equality (\ref{eq:logic}) may not be true:
    \begin{equation}\label{eq:logic}
    \left\{\bar{g}\in \G^*_{\text{met}} : \varphi(\bar{g})\text{ holds in }\G^*_{\text{met}} \right\} = \frac{\prod_{n\in\N} \left\{g_n\in G_n : \varphi(g_n)\text{ holds in }G_n \right\}}{\E},
    \end{equation}
    We give a counterexample to (\ref{eq:logic}) in (\ref{torsion2}) below. We consider torsion elements, that is 
    $\varphi_m(\bar{g})=({\bar{g}}^m=e)$, for $m\in\N$.
    
    \item \label{torsion} Fix a metric group $(G,\|\cdot\|)$, $\varepsilon \geq 0$ and $m\in\N$. We define \emph{elements of $\varepsilon$-order $m$} as elements from $T_{m,\le\varepsilon}(G)$ where : 
    \begin{equation}\label{eq:tors}
    T_{m,\leq\varepsilon}(G) = \{g\in G: \|g^m\|\leq \varepsilon\}\text{, also define }
    T_{m,\varepsilon}(G) = \{g\in G: \|g^m\|<\varepsilon\}.
    \end{equation}
    Those set may not be internal, however we have as in (\ref{balls}):
    \begin{equation}\label{eq:subset}
    T_{m,\varepsilon}\left(\G^*_{\text{met}}\right) \subseteq \frac{\prod_{n\in\N} T_{m,\varepsilon}(G_n)}{\E} \subseteq \frac{\prod_{n\in\N} T_{m,\le \varepsilon}(G_n)}{\E} \subseteq   T_{m,\le\varepsilon}\left(\G^*_{\text{met}}\right).
    \end{equation}

    \item \label{torsion2}
    The last $\subseteq$ in (\ref{eq:subset}) above may be a proper subset, as we give an example where $T_{3,\le0}(\G^*_{\text{met}})\neq\emptyset$, but $T_{n,3,\le 0}(G_n)=\emptyset$. Take \[G_n=\left(\Z_{2^n},\|\cdot\|_{2^n}\right),\] where $\|\cdot\|_{2^n}\colon \Z_{2^n}\to [0,1]$ is a natural norm, given by $\|g\|_{2^n}=\frac{1}{2^{n-1}}\min\{g,2^n-g\}$ (called \emph{Lee norm}), for $g\in\{0,1,\ldots,2^n-1\}$.
    Consider $\varphi_3(g)=(3g=0)$.
    Then each $G_n$ has no elements of order 3, but $\G^*_{\text{met}}$ has such element, e.g. $\bar{g} = \left(\left\lfloor \frac{2^n}{3} \right\rfloor\right)_{n\in\N}/\E$ has order 3.
\end{enumerate}
\end{example}

Let us notice that metrically internal set are closed under $\cdot$ and $\cup$.

\begin{remark}\label{rem:clos}
\begin{enumerate}
     \item  If $Y=\frac{\prod_{n\in\N} Y_n}{\E}$ and $Z=\frac{\prod_{n\in\N} Z_n}{\E}$ are metrically internal, then so are \[Y\cdot Z = \frac{\prod_{n\in\N} Y_n\cdot Z_n}{\E}\text{ and }Y\cup Z = \frac{\prod_{n\in\N} Y_n\cup Z_n}{\E}.\]
    
    \item $\G^*_{\text{met}}\setminus X$ may not be metrically internal, for metrically internal $X =  \frac{\prod_{n\in\N} X_n}{\E}$ (for example $\G^*_{\text{met}}\setminus \{e\}$ is not metrically internal for compact $G$, see Example \ref{ex:compact}). However, this set is contained in a canonical metrically internal subset
    $\G^*_{\text{met}}\setminus X \subseteq 
    \frac{\prod_{n\in\N} G_n\setminus X_n}{\E}$.
    \end{enumerate}
\end{remark}


Lemma \ref{lem:inter} below is crucial in the proof of Theorem \ref{thm:sim2}, which is new  analogue of the classical compactness theorem for $\G^*_{\text{met, fin}} = \left\{x\in \G^*_{\text{met}} : \|x\|<\infty\right\}$.

\begin{lemma} \label{lem:inter}
If $X\subseteq \G^*_{\text{met}}$ is metrically internal, then $X\cap \G^*_{\text{met, fin}}$ is closed with respect to the topology on $\G^*_{\text{met, fin}}$ induced by $\|\cdot\|$.
\end{lemma}
\begin{proof} 
Let $X= \prod_{n\in\N} X_n/\E$. It suffices to prove that if $a_0,a_1, a_2,\ldots$ is a sequence from $X\cap \G^*_{\text{met, fin}}$, converging to $a_\infty\in G^*_{\text{met, fin}}$, then $a_\infty\in X$. For each $n\in\N_{>0}$ take ${k_n}\in\N$ such that 
\begin{equation}\label{eq:l}
\left\|a_{k_n}^{-1}a_{\infty}\right\|< \frac{1}{n}.
\end{equation}
Write $a_{k_n} = \left(a_{k_n,0},a_{k_n,1},a_{k_n,2},a_{k_n,3},\ldots\right)\E$ and $a_{\infty} = \left(a_{\infty,0},a_{\infty,1},a_{\infty,2},a_{\infty,3},\ldots\right)/\E$, for some $a_{k_n,m}\in X_m$, $a_{\infty,m}\in G_m$. For each $n\in\N$ define \[U_n=\left\{m\in\N : \left\|a_{k_n,m}^{-1}a_{\infty,m}\right\|_m<\frac{1}{n}\text{ in }G_m\right\}.\] Every $U_n$ belongs to $\U$, by (\ref{eq:l}). Define \[V_n=U_0\cap U_1\cap\cdots \cap U_n\setminus \{n\}\in\U.\] Then $V_0\supseteq V_1\supseteq V_2\supseteq\ldots\supseteq V_n\in\U$ and $\bigcap_{n\in\N}V_n = \emptyset$, since $n\not\in V_n$. Define
\begin{equation}\label{eq:def}
b_{m} = 
\begin{cases}	
a_{k_n,m} & \text{: when }m\in V_{n}\setminus V_{n+1} ,\\ 
e & \text{: when } m\not\in V_0\end{cases}.
\end{equation}
Define $b_{\infty} = (b_{0},b_{1},b_{2},b_{3},\ldots)/\E$. Then  $b_m=a_{k_n,m}\in X_m$, so $b_{\infty}\in X$. Our aim it to prove that $a_{\infty}=b_{\infty}$.
It is enough to prove that  $\left\|a_{k_n}^{-1}b_{\infty}\right\| < \frac{2}{n}$ holds for every $n\in\N$ (since $a_{\infty} = \lim_{n\to\infty} a_{n}$). Fix $n\in\N$ and take any $m\in V_n$. Then $m\in V_t\setminus V_{t+1}$ for some $t\geq n$, so $b_m=a_{k_t,m}$ by (\ref{eq:def}). Hence \[\left\|a_{k_n,m}^{-1}b_m\right\|_m = \left\|a_{k_n,m}^{-1}a_{k_t,m}\right\|_m \leq \left\|a_{k_n,m}^{-1}a_{\infty,m}\right\|_m + \left\|a_{\infty,m}^{-1}a_{k_t,m}\right\|_m < \frac{1}{n}+\frac{1}{t}\leq\frac{2}{n}.\] 
Therefore $\left\|a_{k_n}^{-1}b_{\infty}\right\| = \lim_{m\to \U}\left\|a_{k_n,m}^{-1}b_m\right\|_m < \frac{2}{n}$.
\end{proof}

For compact groups, Lemma \ref{lem:inter} gives a complete description of metrically internal sets, as explained in Example below.

\begin{example}\label{ex:compact}
Suppose $(G,\|\cdot\|)$ is a compact metric group (that is $\|\cdot\|$ induces a compact topology on $G$). Then $\|\cdot\|$ is a bounded function and $G^*_{\text{met}} = G^*_{\text{met, fin}}$. Furthermore $G^*_{\text{met}} \cong G$ are isomorphic as metric groups (see \cite[Ex. 10.44]{drutu}). In this case we have that 
\[X\subseteq G^*_{\text{met}}=G\text{ is metrically internal }\Longleftrightarrow X\text{ is a closed subset in } G.\]
Indeed, $\Rightarrow$ is by Lemma \ref{lem:inter},  $\Leftarrow$ follows from the fact that $X = \frac{X\times X\times X\times\cdots}{\E}$ for any closed subset $X$ of compact $G$. 
\end{example}

\subsection{New compactness theorem}
\label{sec:comp}

We fix throughout this section a family of metric groups $\G = (G_n,\|\cdot\|_n)_{n\in \N}$ and a countable family $\{X_m\}_{m\in\N}$ of metrically internal (Definition \ref{def:inter}) subsets of $\G^*_{\text{met}}$, where
\begin{equation}\label{def:metinr}
X_m =  \frac{X_{m,0}\times X_{m,1}\times X_{m,2}\times\cdots}{\E},\ \ X_{m,n}\subseteq G_n.
\end{equation}
We give a new compactness theorem for $\G^*_{\text{met, fin}}$. Our context is when $\G^*_{\text{met, fin}}$ is covered by the union of $\{X_m\}_{m\in\N}$. We then find some kind of finite sub-cover. Let us first recall a well known compactness theorem for standard (discrete) ultraproduct $\G^*$ of $\G$ (see (\ref{def:ult})).

\begin{theorem}\label{comp:st}
The following facts are equivalent:
\begin{enumerate}
    \item $\G^* = \bigcup_{m\in\N} X_m$
    \item There is $N\in\N$ such that $\G^* = X_0\cup \ldots \cup X_N$, which is equivalent with the condition: for $\U$-almost all $n\in\N$
\begin{equation}\label{eq:com}
G_n = X_{0,n} \cup \ldots \cup X_{N,n} 
\text{ holds in }G_n.
\end{equation}
\end{enumerate}
\end{theorem}

Below is a generalization of Theorem \ref{comp:st} to the metric setting. We first formulate a result assuming that $\{\|\cdot\|_n\}_{n\in\N}$ is uniformly bounded, that is $\G^*_{\text{met, fin}} = \G^*_{\text{met}}$.

\begin{theorem} \label{thm:sim_b}
Suppose $\sup\{\|\cdot\|_n : n\in\N\}<\infty$ (then $\G^*_{\text{met, fin}} = \G^*_{\text{met}}$). The following conditions are equivalent.
\begin{enumerate}
\item $\G^*_{\text{met}} = \bigcup_{m\in\N} X_m$
\item For any countable infinite sequence of positive reals $(\varepsilon_0,\varepsilon_1,\ldots)\subset\R_{>0}$ there is $N\in\N$ such that
\begin{equation}\label{eq:c}
\G^*_{\text{met}} =  X_{0} \B_{\varepsilon_0}(e) \cup \ldots \cup X_{N} \B_{\varepsilon_N}(e),\text{ holds in }\G^*_{\text{met}},
\end{equation}
which is equivalent with the clause: for $\U$-almost all $n\in\N$
\begin{equation}\label{eq:compp}
G_n =  X_{0,n} \B_{\varepsilon_0}(e) \cup \ldots \cup X_{N,n} \B_{\varepsilon_N}(e),\text{ holds in }G_n.
\end{equation}
\end{enumerate}
\end{theorem}

We give a general version of our compactness theorem (unbounded case).

\begin{theorem} \label{thm:sim2}
The following conditions are equivalent.
\begin{enumerate}
\item $\G^*_{\text{met, fin}} \subseteq \bigcup_{m\in\N} X_m$

\item For every $t>0$ and any infinite sequence of positive reals $(\varepsilon_0,\varepsilon_1,\ldots)\subset\R_{>0}$ there is $N\in\N$ such that $(\varepsilon_0,\varepsilon_1,\ldots,\varepsilon_N)$ has the following property (\ref{eq:comp}): for $\U$-almost all $n\in\N$
\begin{equation}\label{eq:comp}
\B_{t}(e)\subseteq X_{0,n} \B_{\varepsilon_0}(e) \cup \ldots \cup X_{N,n} \B_{\varepsilon_N}(e)
\text{ holds in }G_n.
\end{equation}
\item For every $t>0$ and any infinite sequence of positive reals $(\varepsilon_0,\varepsilon_1,\ldots)\subset\R_{>0}$ there is $N\in\N$ such that
\begin{equation}\label{eq:co}
\B_{t}(e)\subseteq X_{0} \B_{\varepsilon_0}(e) \cup \ldots \cup X_{N} \B_{\varepsilon_N}(e)
\text{ holds in }\G^*_{\text{met, fin}}.
\end{equation}
\end{enumerate}
\end{theorem}

\begin{proof}
$(1)\Rightarrow(2)$ 
Assume that (2) is not true. Let $t>0$ and $(\varepsilon_m)_{m\in\N}$ be  counterexamples. For any $N \in \N$ define 
\begin{equation}\label{def:uu}
U_N=\{n \in \N: (\varepsilon_0, \varepsilon_1, \ldots, \varepsilon_N) \text{ fails to satisfy (\ref{eq:comp}) in } G_n \} \in \U.
\end{equation}
For any $k\in \N$ and $n \in U_k$ let $g_{k,n}\in G_n$ be such that $\|g_{k,n}\|_n < t$ and
\begin{equation}\label{eq:gg}
g_{k,n} \not \in X_{0,n} \B_{\varepsilon_0}(e) \cup \ldots \cup X_{k,n} \B_{\varepsilon_k}(e) \text{ in }G_n.
\end{equation}

Clearly $U_0\supseteq U_1\supseteq U_2\supseteq\ldots$. We may assume that $\bigcap \{U_k: k\in \N \}=\emptyset$ (just replace $U_k$ by $U_k\setminus \{0,1,\ldots,k\}$, which is still in $\U$). For each $n \in U_0$ let $k_n\in\N$ be the greatest $k\in\N$ such that $n \in U_k$, that is $n\in U_{k_n}$, which means that $(\varepsilon_0, \varepsilon_1, \ldots, \varepsilon_{k_n})$ fails to satisfy (\ref{eq:comp}) in $G_n$.

Then $\lim_{n\to\U}k_n=\infty$. Define
\[g = (g_{k_n,n})_{n\in\N}/\E\in \G^*_{\text{met}}.\]
Then $\|g\|\le t$, so $g\in \G^*_{\text{met, fin}}$. However $g\not\in \bigcup_{m\in\N} X_m$, which gives the contradiction with (1). Indeed, we prove that $g\not\in X_m$ for every $m\in\N$. (\ref{def:uu}) and (\ref{eq:gg}) imply that \[g_{k_n,n}\not\in X_{m,n}\B_{\varepsilon_m}(e)\] holds for all $n\in\N$ such that $k_n>m$. Therefore the distance between $g$ and $X_m$ is at least $\varepsilon_m>0$, so $g\not\in X_m$.

$(2)\Rightarrow (3)$ is immediate by Definitions \ref{def:inff} and \ref{def:fin}.

$(3)\Rightarrow (1)$ Assume (1) fails. Then there is a non-trivial $g=(g_m)_{m\in\N}/\E \in \G^*_{\text{met, fin}}$ such that $g\not\in \bigcup_{m\in\N} X_m$. This means that $\left\|g y^{-1}\right\| > 0$, for all $y\in X_m$. Since each $X_m$ is a closed set (by Lemma \ref{lem:inter}), there is $\varepsilon>0$ such that $\left\|g y^{-1}\right\| > \varepsilon$ for all $y\in X_m$, $m\in\N$. Hence define \[\varepsilon_m = \inf\left\{\left\|g y^{-1}\right\| : y\in X_m \right\}>0.\] 
Then (3) fails for $t:=\|g\|+1>0$ and $(\varepsilon_m)_{m\in\N}$. Indeed, suppose that $(\varepsilon_0, \varepsilon_1, \ldots, \varepsilon_N)$ has property (\ref{eq:co}) for some $N\in\N$. Then $\left\|g y^{-1}\right\|< \varepsilon_m$, for some $m\le N$ and $y\in X_m$, contradiction.
\end{proof}

There is a further generalization of Theorem \ref{thm:sim2}, where instead of $\G^*_{\text{met, fin}}$ one can put any metrically internal set $Y$. That is, one can characterize the situation when $Y\cap \G^*_{\text{met, fin}}$ is covered by a countably many metrically internal sets as a property of $\U$-almost all coordinates.

We do not need this stronger version of \ref{thm:sim2}, hence we only state this in Remark \ref{rem:gen} below. The proof of Remark \ref{rem:gen} can be easily derived from the proof of Theorem \ref{thm:sim2}.

\begin{remark} \label{rem:gen}
Suppose $Y =  \frac{Y_{0}\times Y_{1}\times Y_{2}\times\cdots}{\E}$ is a metrically internal subset of $\G^*_{\text{met}}$. The following conditions are equivalent under the notation from Theorem \ref{thm:sim2}:
\begin{enumerate}
\item $Y\cap \G^*_{\text{met, fin}} \subseteq \bigcup_{m\in\N} X_m$

\item For every $t>0$ and any infinite sequence of positive reals $(\varepsilon_0,\varepsilon_1,\ldots)\subset\R_{>0}$ there is $N\in\N$ such that $(\varepsilon_0,\varepsilon_1,\ldots,\varepsilon_N)$ has the following property (\ref{eq:comppp}) in $G_n$, for $\U$-almost all $n\in\N$:
\begin{equation}\label{eq:comppp}
Y_n\cap \B_{t}(e)\subseteq X_{0,n} \B_{\varepsilon_0}(e) \cup \ldots \cup X_{N,n} \B_{\varepsilon_N}(e)
\text{ holds in }G_n.
\end{equation}
\end{enumerate}
\end{remark}

\section{Corollaries of compactness theorem} 
\label{sec:baire}

We give a couple of consequences of our compactness theorems \ref{thm:sim_b} and \ref{thm:sim2}, mainly in the \emph{bounded case}, that is when:
\begin{equation}\label{bounded}
\sup\{\|\cdot\|_n : n\in\N\}<\infty.
\end{equation}
This condition implies that $\G^*_{\text{met, fin}} = \G^*_{\text{met}}$. We write explicitly \emph{bounded case}, when (\ref{bounded}) is assumed.

\subsection{Square of finite subcover}

Lemma \ref{lem:propp} (2) asserts that $(\G^*_{\text{met, fin}}, \|\cdot\|)$ is a complete metric space, so we can apply Baire category argument. We give another consequence of cover condition (1) from Theorem \ref{thm:sim_b}.
\begin{theorem}\label{thm:ccomp} (bounded case) 
If $\G^*_{\text{met}} = \bigcup_{m\in\N} X_m$, where each $X_m$ is metrically-internal, then there is $N\in\N$ such that
\begin{equation}\label{eq:22}
\G^*_{\text{met}} = \left(X_0 \cup\ldots\cup X_N \right)^2.
\end{equation}
\end{theorem}

Let us observe that (\ref{eq:22}) from Theorem \ref{thm:ccomp} cannot be simplified to apparently simpler condition $\G^*_{\text{met}} \subseteq X_1 \cup\ldots\cup X_N$, as shown in Example \ref{ex:cover}.

\begin{proof} 
$(\G^*_{\text{met}}, \|\cdot\|)$ is a complete metric space by Lemma \ref{lem:propp} (2). Each $X_m\subseteq  \G^*_{\text{met}}$ is closed subset by Lemma \ref{lem:inter}. By Baire category theorem, some $X_m$ has a non-empty interior, that is there is $m\in\N$, $\varepsilon>0$ and $g\in X_m$ such that \begin{equation}
    \B_\varepsilon(g) \subseteq X_m,\text{ hence }\B_\varepsilon(e)\subseteq X_m g^{-1}.
\end{equation}
By applying Theorem \ref{thm:sim_b} to $(\varepsilon,\varepsilon,\varepsilon,\ldots)$, we get $N'\in\N$ such that  
\begin{equation}\label{eq:cc}
\G^*_{\text{met}} =  X_{0} \B_{\varepsilon}(e) \cup \ldots \cup X_{N'} \B_{\varepsilon}(e)=\left(X_0 \cup\ldots\cup X_{N'} \right)\cdot X_m\cdot g^{-1}.
\end{equation}
Whence $\G^*_{\text{met}} =\left(X_0 \cup\ldots\cup X_{N'} \right)\cdot X_m$. The conclusion is true for $N=\max\{N',m\}$.
\end{proof}

\subsection{Bounded generation, torsion, perfectness and simplicity}

Standard (discrete) ultraproduct $\G^*$ is \emph{saturated} in model-theoretic sense. This fact has many immediate corollaries around \emph{uniform group properties}. Let us remind that by a \emph{commutator} $\left[g,h\right]$ we mean $g^{-1}h^{-1}gh$. A group $G$ is \emph{perfect}, if every element of $G$ is a product of commutators. Here are well known facts about standard ultraproduct $\G^*$:
\begin{itemize}
    \item if $\G^*$ is a perfect group, then $\G^*$ is \emph{uniformly perfect} (Corollary \ref{cor:perf});
    
    \item if $\G^*$ is a simple group, then $\G^*$ must be \emph{uniformly simple} (Definition \ref{def:bsim} (2));
    
    \item if $\G^*$ is a torsion group (i.e. for every element $g\in \G^*$ there is $N\in\N$, such that $g^N=e$), then $\G^*$ is \emph{uniformly torsion} (also called \emph{of finite exponent}), that is there is $N\in\N$ such that $g^N=e$, for all $g\in \G^*$.
\end{itemize}

   We derive below generalization some of these facts to metric ultraproduct $\G^*_{\text{met}}$. 
   
   \subsubsection{Bounded generation}\label{sec:bg}
   
We first give a general fact on \emph{bounded generation}.

\begin{lemma}\label{lem:gen} (bounded case)
Suppose $X\subseteq \G^*_{\text{met}}$ is a metrically internal subset. If $X$ generates $\G^*_{\text{met}}$, then $X$ generates in finitely many steps, i.e. there is $N\in\N$ such that \[\G^*_{\text{met}} = \left(X\cup X^{-1}\right)^{N}.\]
\end{lemma}
\begin{proof}
Define $X_m = \left(X\cup X^{-1}\right)^{m}$. Then $\G^*_{\text{met}} = \bigcup_{m\in\N} X_m$, so the conclusion follows from Theorem \ref{thm:ccomp}.
\end{proof}

We now derive a couple of corollaries. 

\subsubsection{Uniform perfectness}\label{sec:perf}

An immediate consequence of Lemma \ref{lem:gen} is the following corollary.

\begin{corollary}\label{cor:perf} (bounded case)
Suppose $\G^*_{\text{met}}$ is perfect, then $\G^*_{\text{met}}$ is uniformly perfect, i.e. there is $N\in\N$ such that every element of $\G^*_{\text{met}}$ is a product of $N$ commutators.
\end{corollary}
\begin{proof}
The conclusion follows from Lemma \ref{lem:gen}, as by Example \ref{ex:internal} (\ref{ex:comm}), $X=\left\{\left[\bar{g}, \bar{h}\right] : \bar{g},\bar{h}\in \G^*_{\text{met}}\right\}$ is metrically internal.
\end{proof}

Obviously if $\G^*$ is a perfect group, then so is $\G^*_{\text{met}}$, as $\G^*_{\text{met}}$ is a homomorphic image of $\G^*$ (Remark \ref{rem:xxx} (2)). The converse is not true in general, there is non-perfect $\G^*$ with perfect $\G^*_{\text{met}}$, see Example \ref{ex:dos}. Therefore Corollary \ref{cor:perf} cannot be obtained by using only classical compactness theorem \ref{comp:st} applied to $\G^*$.
 

\subsubsection{Uniform torsion}\label{sec:tors}

Let us now consider torsion groups.  

A well known fact of standard ultraproduct $\G^*$ is: $\G^*$ is torsion if and only if $\G^*$ is uniformly torsion. We conjecture that the same is true for any metric ultraproduct $\G^*_{\text{met}}$.

\begin{conjecture}\label{con:tor}
If $\G^*_{\text{met}}$ is a torsion group, then $\G^*_{\text{met}}$ has finite exponent (i.e. is uniformly torsion).
\end{conjecture}


We are able to prove an $\varepsilon$-analogue of Conjecture \ref{con:tor} for $\G^*_{\text{met}}$, provided that $\varepsilon>0$.

\begin{definition}\label{def:etors} Fix $\varepsilon > 0$ and metric group $(G,\|\cdot\|)$.
\begin{enumerate}
    \item $G$ is called \emph{$\varepsilon$-torsion}, if for every $g\in G$ there is $N\in\N$ such that $\|g^N\|< \varepsilon$. A stronger notion than (1) is (2):
    \item $G$ is called \emph{uniformly $\varepsilon$-torsion} (or \emph{$\varepsilon$-finite exponent}) if there is $N\in\N$, such that $\|g^N\|< \varepsilon$ for all $g\in G$. A bit weaker that (2), but still stronger that (1) is (3) below.
    \item $G$ is called \emph{almost uniformly $\varepsilon$-torsion} if there is $N\in\N$, such that for every $g\in G$ there is $m\le N$ with $\|g^m\|< \varepsilon$.
\end{enumerate}
\end{definition}

Let us apply Theorem \ref{thm:sim_b} together with Example \ref{ex:internal} (\ref{torsion}).

\begin{corollary}\label{cor:tors} (bounded case) Fix $\varepsilon>0$.
If $\G^*_{\text{met}}$ is $\varepsilon$-torsion group, then $G$ is  almost uniformly $2\varepsilon$-torsion.
\end{corollary}

Observe that a circle group $\S=(S^1,\cdot)$ is a non-torsion group, but $\varepsilon$-torsion for every $\varepsilon>0$. Moreover $\S^*_{\text{met}} = \S$ (by Example \ref{ex:compact}) is not uniformly $\varepsilon$-torsion, but almost uniformly $\varepsilon$-torsion, for every $\varepsilon>0$. Hence Corollary \ref{cor:tors} cannot be improved to uniform $\varepsilon$-torsion.  

\begin{proof} 
Let $X_m = \frac{\prod_{n\in\N} T_{m, \varepsilon}(G_n)}{\E}$, for $m\in\N_{>0}$ be metrically internal sets from (\ref{eq:tors}) in Example \ref{ex:internal} (\ref{torsion}). Then  $\G^*_{\text{met}} = \bigcup_{m\in\N_{>0}} X_m$, since $\G^*_{\text{met}}$ is $\varepsilon$-torsion. Let us apply (\ref{eq:c}) from Theorem \ref{thm:sim_b} to $(\varepsilon,\frac{\varepsilon}{2},\frac{\varepsilon}{3},\frac{\varepsilon}{4},\ldots)$. Then there is $N\in\N$ such that 
\[\G^*_{\text{met}} = \bigcup_{m=1}^N \B_{\frac{\varepsilon}{m}}(e) X_m.\] Take any $g\in \G^*_{\text{met}}$, then there is $m\le N$ and $b\in \B_{\frac{\varepsilon}{m}}(e)$, $x\in X_m$ such that $g=bx$. Then $\|g^m\|=\|(bx)^m\|=\|b^m x^m\|\le \|b^m\| + \|x^m\|< m\|b\| + \varepsilon< 2\varepsilon$. Hence $\|g^m\|<2\varepsilon$.
\end{proof}

\begin{remark}
Many examples of metric groups we consider do satisfy the following property:
\begin{equation}\label{zalozenie}
    \|g^n\|\le \|g\|, \text{ for any }g\in G,\ n\in\N.
\end{equation}
In particular (\ref{zalozenie}) holds for:
\begin{itemize}
    \item permutation groups with the Hamming norm $(S_n,\|\cdot\|_H)$ (Definition \ref{def:hamm}),
    \item finite groups with conjugacy length \cite[Lemma 2.5]{thoms}, \cite[Theorem 1.1]{liesha}, which is a pseudo-norm: \[\|g\|_c=\frac{\log\left(\left|g^G\right|\right)}{\log\left(|G|\right)},\]
    \item linear groups with the Jordan length: see (\ref{jordan}) in Section \ref{sec:lin}.
\end{itemize}
If a metric group $(G,\|\cdot\|)$ satisfies (\ref{zalozenie}), then almost uniform $\varepsilon$-torsion implies uniform {$\varepsilon$-torsion} from Definition \ref{def:etors}. Indeed, if $N$ satisfies (3) from Definition \ref{def:etors}, then $\left\|g^{N!}\right\|<\varepsilon$, for all $g\in G$.
\end{remark}


\subsubsection{Uniform and bounded simplicity} \label{sec:ub}

Let us consider \emph{simplicity} and related stronger properties.

\begin{definition} \label{def:bsim}
\begin{enumerate}
    \item A group $G$ is called \emph{boundedly simple} if for any $g \in G\setminus\{e\}$ there is a natural number $N$ such that 
    \begin{equation}\label{eq:cn}
    C_N(g,G) := \left(g^{G} \cup g^{-1 G}\right)^{\leq N} = G.
    \end{equation}
    In other words $C_N(g,G)$ is the set of all products of at most $N$ conjugates of $g$ and $g^{-1}$.
    \item $G$ is call \emph{uniformly simple} if there exist $N\in\N$ such that $C_N(g,G)=G$ for every $g\in G\setminus\{e\}$. We say then that $G$ is \emph{$N$-uniformly simple}.
\end{enumerate}
\end{definition}

Bounded simplicity appeared in the literature in many places (see \cite[Section 1]{gg}) and under different names, for example as \emph{bounded normal generation} \cite{dt}. It is a well know fact that discrete ultraproduct $\G^*$ is simple if and only if $\G^*$ is uniformly simple. This is not true for $\G^*_{\text{met}}$ (see Example \ref{ex:per}). In a metric case we need to switch to bounded simplicity.

\begin{corollary}\label{cor:bbsim} (bounded case)
$\G^*_{\text{met}}$ is simple if and only if $\G^*_{\text{met}}$ is boundedly simple.
\end{corollary}
\begin{proof}
By Example \ref{ex:internal} (\ref{conj}) every nontrivial conjugacy class $g^{\G^*_{\text{met}}}$ is metrically internal and generates $\G^*_{\text{met}}$ (since $\G^*_{\text{met}}$ is simple). The conclusion follows from Lemma \ref{lem:gen}.
\end{proof}

A further study of simplicity $\G^*_{\text{met, fin}}$, also in unbounded case (under some mild Assumption \ref{eq:ass}) is contained in Section \ref{sec:sim}.

\section{Simplicity} \label{sec:sim}

This section is devoted the general study of $\G^*_{\text{met, fin}}$. We do not longer assume that norms are uniformly bounded. We extend Corollary \ref{cor:bbsim} and give a general criterion for simplicity of $\G^*_{\text{met, fin}}$, under an Assumption \ref{eq:ass} below.

Let $\G = (G_n,\|\cdot\|_n)_{n\in \N}$ be a family of metric groups. We need to assume the following condition (Assumption \ref{eq:ass}), in order to give a smooth criterion for simplicity of metric ultraproducts of metric groups. 
This assumption is clearly satisfied when $\{\|\cdot\|_n\}_{n\in\N}$ are uniformly bounded; that is when $\sup\left\{\|g\|_n : n\in\N, g\in G_n\right\}<\infty$. However, there are many important family of unbounded metric groups which do satisfy (\ref{eq:ass}). Intuitively (\ref{eq:ass}) says that conjugacy class of any $g\in \bigcup_{n\in\N} G_n$ can be determined by uniformly $\|\cdot\|$-short elements.


\begin{ass}\label{eq:ass}
There is a non-decreasing function $F_\G\colon \R_{>0} \to \R_{>0}$ for a family of pseudo-metric groups $\G=(G_n,\|\cdot\|_n)_{n\in \N}$ such that, for every $n\in\N$ and any $g\in G_n$, 
\[g^{G_n} = \left\{ h^{-1}g h : h\in G_n,\ \|h\|_n< F_\G\left(\left\|g\right\|_n\right) \right\}.\]
\end{ass}

If norms are uniformly bounded, then Assumption \ref{eq:ass} is true.

We use the following immediate application of Assumption \ref{eq:ass}, which says that  products of conjugacy classes of elements from $\G^*_{\text{met, fin}}$ computed in a bigger group $\G^*_{\text{met}}$ behaves well after intersecting them with $\G^*_{\text{met, fin}}$.

\begin{fact}\label{fact:ass}
Under Assumption \ref{eq:ass} the following holds: for any $\bar{g}\in \G^*_{\text{met, fin}}$ and any $n\in\N$
\[\bar{g}^{\G^*_{\text{met}}} = \bar{g}^{\G^*_{\text{met, fin}}},\ \ C_n\left(\bar{g}, \G^*_{\text{met}}\right) = C_n\left(\bar{g}, \G^*_{\text{met, fin}}\right)\] 
(where $C_n(g,G) = \left(g^G\cup {g^{-1}}^G\right)^{\leq n}$ for $n\geq 1$, and $C_0(g,G) =\{e\}$, according to (\ref{eq:cn})).
Therefore $C_n\left(\bar{g}, \G^*_{\text{met, fin}}\right)$ is a closed subset of $\G^*_{\text{met, fin}}$ by Example \ref{ex:internal} and Lemma \ref{lem:inter}.
\end{fact}
\begin{proof}
It is enough to prove that $\bar{g}^{\G^*_{\text{met}}} \subseteq \bar{g}^{\G^*_{\text{met, fin}}}$, as the rest follows by Remark \ref{rem:clos}. Take $\bar{x} \in \bar{g}^{\G^*_{\text{met}}}$, then $\bar{x}=(x_n)_{n\in\N}/\E = \left(g_n^{h_n}\right)_{n\in\N}/\E$, where $\|h_n\|_n < F_{\G}(\|\bar{g}\|)$. Hence \[\bar{h}=(h_n)_{n\in\N}\in \G^*_{\text{met, fin}},\] so $\bar{x}\in \bar{g}^{\G^*_{\text{met, fin}}}$.
\end{proof}

\subsection{Characterization of simplicity}

Fix a family $\G = (G_n,\|\cdot\|_n)_{n\in \N}$ of metric groups and non-principal ultrafilter $\U$. The following theorem gives a characterization of simplicity for finitary subgroup  $\G^*_{\text{met,fin}}$ (see (\ref{def:fin})).

\begin{theorem} \label{thm:SimpFam} 
The following facts are equivalent under Assumption \ref{eq:ass}.
\begin{enumerate}
\item Finitary metric ultraproduct $\G^*_{\text{met, fin}}$ is a simple group.

\item For all $t>r>0$ and for every infinite sequence of positive reals $(\varepsilon_1,\varepsilon_2,\ldots)$ there is $N\in\N$ such that for $\U$-many $k\in\N$, for every 
$g\in G_k$, such that $\|g\|_k> r$ \begin{equation}\label{eq:777}
\B_t(e) \subseteq \bigcup_{n=1}^N   C_n\left(g,G_k\right) \B_{\varepsilon_n}(e)   \text{ holds in }G_k.
\end{equation}

\item For all $t>r>0$ and for every infinite sequence of positive reals $(\varepsilon_1,\varepsilon_2,\ldots)$ there is $N\in\N$ such that for every 
$\bar{g}\in \G^*_{\text{met, fin}}$, $\|\bar{g}\|> r$ 
\begin{equation} \label{cond2}
\B_t(e) \subseteq \bigcup_{n=1}^N  C_n\left(\bar{g},\G^*_{\text{met}} \right) \B_{\varepsilon_n}(e)  \text{ holds in }\G^*_{\text{met}}.
\end{equation}



%
\end{enumerate}
\end{theorem}

We use the following lemma.

\begin{lemma}\label{lem:aluk}
Let $(G,\|\cdot\|)$ be a metric group, $g,h\in G$, $n\in\N$ and $\varepsilon > \|g^{-1} h\|$. Then
\[C_n(h,G) \subseteq C_n(g,G) \B_{n\cdot \varepsilon}(e).\]
In other words, \[\bigcup \left\{  C_n(h,G) : h\in \B_\varepsilon(g) \right\} \subseteq C_n(g,G) \B_{n\cdot \varepsilon}(e). \]
\end{lemma}
\begin{proof}
Take $x\in C_n(h,G)$. Then $x=h^{\pm y_1}\cdot\ldots\cdot h^{\pm y_n}$, for some $y_1,\ldots,y_n\in G$. Define $t=g^{-1} h$, then $\|t\|<\varepsilon$ and $x=g^{\pm y_1}\cdot t^{\pm y_1}\cdot\ldots\cdot g^{\pm y_n}\cdot t^{\pm y_n}=g^{\pm z_1}\cdot\ldots\cdot g^{\pm z_n}\cdot t^{\pm c_1}\cdot\ldots t^{\pm c_n} \in C_n(g,G) \B_{n\cdot \varepsilon}(e)$, for some $z_1,\ldots,z_n,c_1,\ldots,c_n\in G$.
\end{proof}

\begin{proof}[Proof of Theorem \ref{thm:SimpFam}]
Define a family of metric groups $\H=(G_n\times G_n,\|\cdot\|^2_n)$, where \[\|(a,b)\|^2_n = \max\{\|a\|,\|b\|\} \ \ a,b\in G_n.\] Observe that a ball $\B_\varepsilon(a,b)$ in $G_n\times G_n$ is of the form $\B_\varepsilon(a)\times \B_\varepsilon(b)$, for balls $\B_\varepsilon(a)$, $\B_\varepsilon(b)$ in $G_n$. Moreover
$\H^*_{\text{met, fin}} = \G^*_{\text{met, fin}} \times \G^*_{\text{met, fin}}$.
Define
\begin{align}
X_0 = \G^*_{\text{met}}\times \{e\},\text{ and for }n>0,\ X_n = \left\{(\bar{x},\bar{y}) : \bar{x}\in C_n\left(\bar{y},\G^*_{\text{met}}\right)\right\}, \label{lodom} \\
X_{0,k} = G_k\times \{e\},\text{ and for }n>0,\ X_{n,k} = \left\{(x,y) : x\in C_n\left(y,G_k\right)\right\}.
\end{align}
Clearly $X_0, X_n$, $n\in\N$ are metrically internal subsets of $\H^*_{\text{met}}$, as $X_n =  \frac{X_{n,0}\times X_{n,1}\times X_{n,2}\times\cdots}{\E}$. Observe that $\G^*_{\text{met}}$ is a simple group if and only if  $\{X_n\}_{n>0}$ is a cover of $\G^*_{\text{met, fin}} \times \G^*_{\text{met, fin}} = \H^*_{\text{met,fin}}$: 
\begin{equation}\label{eq:covr}
\H^*_{\text{met,fin}} = \G^*_{\text{met, fin}} \times \G^*_{\text{met, fin}}\subseteq \bigcup_{n\geq 0} X_n.
\end{equation}

$(1)\Rightarrow (2)$ Fix $t>r>0$ and $\bar{\varepsilon}=(\varepsilon_1,\varepsilon_2,\ldots)\subset\R_{>0}$. 
By Theorem \ref{thm:sim2}(2) applied to $(r,\frac{\varepsilon_1}{2},\frac{\varepsilon_2}{3},\frac{\varepsilon_3}{4},\ldots)$ and $\H^*_{\text{met,fin}}$ in  (\ref{eq:covr}), there is $\U$-many $k\in\N$ and $N\in\N$ such that 
\begin{equation}\label{rhs}
\B_t(e) \subseteq X_{0,k} \B_{r}(e)\cup \bigcup_{n=1}^N X_{n,k} \B_{\frac{\varepsilon_n}{n+1}}(e) \text{ holds in }\G_k \times \G_k.
\end{equation}
The right hand side of (\ref{rhs}) is exactly
\begin{align}
\B_t(e) \subseteq  &G_k \times \B_r(e)\cup \bigcup_{n=1}^N\bigcup  \left\{(x,y) : x'\in C_ n\left(y',G_k\right),\ x\in \B_{\frac{\varepsilon_n}{n+1}}\left(x'\right),\ y\in \B_{\frac{\varepsilon_n}{n+1}}\left(y'\right)\right\}.
\end{align}
Lemma \ref{lem:aluk} implies
\begin{align}
\B_t(e) \subseteq  &G_k \times \B_r(e)\cup \bigcup_{n=1}^N  \left\{ (x,y) : x\in C_n\left(y,G_k\right) \B_{\varepsilon_n}(e)\right\} , \label{eq:lul}
\end{align}
Let us prove (2). Suppose $g\in G_k$, $\|g\|_k\geq r$. We prove (\ref{eq:777}). Take $x\in~\B_t(e)$. Then $(x,g)\in G_k \times G_k$ and $g\not\in \B_r(e)$, so by (\ref{eq:lul}), there is $n\in[1,N]$ such that \\ $x\in~C_n\left(g,G_k\right)\B_{\varepsilon_n}(e)$.  

$(2)\Rightarrow (3)$ is standard.

$(3)\Rightarrow (1)$ Take any $\bar{g}\neq e, \bar{h}\in \G^*_{\text{met, fin}}$. It is enough to prove that $\bar{h}$ is a product of conjugates of $\bar{g}$. Take $r=\|\bar{g}\|>0$, $t>\|\bar{h}\|$. Define \[X_{0,5} = \G^*_{\text{met, fin}}\times  \B_r(e).\] We apply Theorem \ref{thm:sim2}(3) to $t$ and a sequence $\bar{\varepsilon}\subset\R_{>0}$. Condition (3) for $t$ and $\bar{\varepsilon}$ implies that
\begin{equation}\label{rhss}
\B_t(e) \subseteq X_{0,5} \B_{\varepsilon_0}(e)\cup \bigcup_{n=1}^N X_n \B_{\varepsilon_n}(e) \text{ holds in }\G^*_{\text{met}} \times \G^*_{\text{met}},
\end{equation}
where $X_n$ is defined in (\ref{lodom}). By Theorem \ref{thm:sim2}, 
\begin{equation}\label{eq:covrrr}
\G^*_{\text{met, fin}} \times \G^*_{\text{met, fin}}\subseteq X_{0,5} \cup \bigcup_{n\geq 1} X_n.
\end{equation}
Since $\bar{g}\not\in \B_r(e)$, $(\bar{h},\bar{g})\not\in X_{0,5}$, so there is $n\in\N_{>0}$ such that  $(\bar{h},\bar{g})\in X_n$. Hence $\bar{h}\in C_n\left(\bar{g},\G^*_{\text{met}}\right)$.
\end{proof}

We extract from the proof of Theorem \ref{thm:SimpFam} and Lemma \ref{lem:aluk} the following fact.

\begin{remark}\label{covid}
The following conditions are equivalent under Assumption \ref{eq:ass}.
\begin{enumerate}
\item Finitary metric ultraproduct $\G^*_{\text{met, fin}}$ is a simple group.

 \item  For all $t>r>0$ and for any sequence $(\varepsilon_0, \varepsilon_1,\ldots)\subset \R_{>0}$ there is $N\in\N$ such that for $\U$-many $k\in \N$, for every $g \in G_k$, $\|g\|_k \geq r$ 
    \[
\B_{t}(e) \subseteq \bigcup \left\{ C_n\left(g',G_k\right) \B_{\varepsilon_n}(e) : 0\leq n\leq N,\ \ g'\in B_{\varepsilon_n}(g)  \right\}.
\]

\end{enumerate}
\end{remark}

Let us name to the property (2) from Theorem \ref{thm:SimpFam} for later use.

\begin{definition} \label{def:rt} Let $(G,\|\cdot \|)$ be a metric group and $t>r>0$.
\begin{enumerate}
    \item We say that $(\varepsilon_0,\ldots,\varepsilon_N)\subset\R_{>0}$ is  \emph{$(r,t)$-big for $G$} if for all $g\in G$ with $\|g\|> r$
\[{\B_{t}(e) \subseteq \bigcup_{n=0}^N C_n(g,G) \B_{\varepsilon_n}(e)}\] 
holds in $G$.
    \item We say that $(\varepsilon_0,\ldots,\varepsilon_N)$ is $(r,t)$-small if it is not $(r,t)$-big.
\end{enumerate}
\end{definition}

Let us note the following immediate remark, which gives a simplicity condition regardless the choice of ultrafilter.

\begin{remark} \label{FinComp} The following conditions are equivalent under Assumption \ref{eq:ass}.
\begin{enumerate}
    \item Metric ultrapower $\G^*_{\text{met, fin}}$ of $\G$ is simple, for every non-principal ultrafilter $\U$.
    \item For all $t>r>0$ and for every infinite sequence of positive reals  $(\varepsilon_0,\varepsilon_1,\ldots)$ there is $N\in\N$ such that
    \[E_n=\left\{n\in\N : (\varepsilon_0,\varepsilon_1,\ldots,\varepsilon_N) \text{ is $(r,t)$-big for } G_n\right\}\] has finite complement in $\N$ (big sequences are definded in \ref{def:rt}).
\end{enumerate}
Indeed, if (2) fails, then the family $F=\{\N\setminus E_n\}_{n\in\N}$ consists of infinite sets and $\N\setminus E_n\supseteq \N\setminus E_{n+1}$. Therefore $F$ could be extended to an ultrafilter $\U$ on $\N$, so (1) fails by Theorem \ref{thm:SimpFam}.
\end{remark}

\subsection{A construction based on a single group}

Let us consider the following construction for a metric group $(G,\|\cdot\|)$. We scale $\|\cdot\|$ by positive numbers to get a family of norms, where Theorem \ref{thm:SimpFam} can be applied. Below is an immediate corollary of Theorem \ref{thm:SimpFam}. We apply this result in Example \ref{ex:per}.

\begin{corollary} \label{cor:stozek}
Suppose $(G,\|\cdot\|)$ is a metric group and let $\bar{c}=(c_n)_{n\in\N}\in\R_{>0}$. Consider \[\G = (G,c_n \|\cdot\|)_{n\in\N}.\] Assume that $\G$ satisfies Assumption \ref{eq:ass}. Then the following facts are equivalent.
\begin{enumerate}
    \item $\G^*_{\text{met, fin}}$ is a simple group.
    \item For all $t>r>0$ and for any sequence $(\varepsilon_0, \varepsilon_1,\ldots)\subset \R_{>0}$ there is $N\in\N$ such that for $\U$-many $k\in \N$, $\left(\frac{\varepsilon_0}{c_k},\ldots,\frac{\varepsilon_N}{c_k}\right)$ is $\left(\frac{r}{c_k},\frac{t}{c_k}\right)$-big for $G_k$, that is, for every $g \in G_k$, $\|g\|_k \geq \frac{r}{c_k}$ 
    \[
\B_{\frac{t}{c_k}}(e) \subseteq \bigcup_{n=0}^N  C_n\left(g,G\right) \B_{\frac{\varepsilon_n}{c_k}}(e).
\]
    \item For all $t>r>0$ and for any sequence $\bar{\varepsilon} = (\varepsilon_0, \varepsilon_1,\ldots)\subset \R_{>0}$ there is $N\in\N$ such that for $\U$-many $k\in \N$, for every $g \in G_k$, $\|g\|_k \geq \frac{r}{c_k}$ 
    \[
\B_{\frac{t}{c_k}}(e) \subseteq \bigcup \left\{ C_n\left(g',G\right) \B_{\frac{\varepsilon_n}{c_k}}(e) : 0\leq n\leq N,\ \ g'\in \B_{\frac{\varepsilon_n}{c_k}}(g)  \right\}.
\]
%
\end{enumerate}
\end{corollary}

\subsection{Metric and topological simplicity}\label{sec:top}

The condition (2) from Theorem \ref{thm:SimpFam} for a single metric group $(G,\|\cdot\|)$ (that is, $\G$ consists only of one metric group) does not depend on the choice of ultrafilter $\U$. Therefore, it make sense to have the following definition.

\begin{definition}\label{def:metsim}
A metric group $(G,\|\cdot\|)$ is called \emph{metrically simple} if its ultrapower $G^*_{\text{met, fin}}$ is a simple group, i.e.  $G$ satisfies the condition (2) from Theorem \ref{thm:SimpFam}.
\end{definition}

A topological $G$ group is called \emph{topologically simple} if every nontrivial normal subgroup of $G$ is dense.

\begin{proposition} \label{prop:psim} \mbox{}
\begin{enumerate}
    \item  Every simple compact metric group (e.g. $SO_3(\R)$) is metrically simple (see Definition \ref{def:metsim}).
    \item Metric simplicity implies topological simplicity, that is if $G$ is metrically simple, then $G$ has no closed normal proper subgroups.
\end{enumerate}
\end{proposition}
\begin{proof}
(1) If $(G,\|\cdot\|)$ is a compact metric group, then $\|\cdot\|$ is a bounded function on $G$ and $G = G^*_{\text{met, fin}}$, as every sequence from $G$ is $\E$-equivalent with a constant sequence. Therefore $G^*_{\text{met, fin}}$ is simple.

(2) If $N\lhd G$ is a closed normal subgroup, then every $g\in G\setminus N$ gives \[g^*=(g,g,g,\ldots) \E\in G^*_{\text{met, fin}} \setminus N^*_{\text{met, fin}},\] so $G^*_{\text{met, fin}}$ is not simple, as it has $N^*_{\text{met, fin}}$ as a proper normal subgroup.
\end{proof}

Since $(\G^*_{\text{met, fin}}, \|\cdot \|)$ is a topological group (Lemma \ref{lem:propp} (1)), it makes sense to ask: is there a criterion for topological simplicity of $\G^*_{\text{met, fin}}$? We answer this question in theorem below.

\begin{theorem} \label{prop:psim2} \mbox{}
The following conditions are equivalent 
\begin{enumerate}
\item Finitary metric ultraproduct $\G^*_{\text{met, fin}}$ is a topologically simple group.

 \item For all $t>r>0$ and $\varepsilon>0$ there is $N\in\N$ such that for all $g\in G$ with $\|g\| \in (r,t]$
 \[\B_t(e) \subseteq  C_N(g,G) \B_{\varepsilon}(e) \text{ holds in } G.\]
\end{enumerate}
\end{theorem}

\begin{proof} 
$(1)\Rightarrow (2)$ Suppose (2) fails, that is there are $t>r>0$ and $\varepsilon>0$ such that for all $n\in\N$ there are $g_n,h_n\in G$ with
\begin{enumerate}
    \item $\|g_n\|\in (r,t]$, $\|h_n\|<t$,
    \item $h_n\not\in C_n(g_n,G)\B_{\varepsilon}(e)$.
\end{enumerate}
Take $g^* = (g_1,g_2,g_3,\ldots)\E$ and $h^*=(h_1,h_2,h_3,\ldots)\E$ from $G^*_{\text{met, fin}}$. Then \[h^*\not\in \langle\langle g^* \rangle\rangle \B_{\varepsilon}(e)\supseteq \overline{\langle\langle g^* \rangle\rangle},\] where 
$H=\langle\langle g^* \rangle\rangle$ is the normal subgroup of $G^*_{\text{met, fin}}$ generated by $g^*$. So $H$ is a proper normal closed subgroup of $G^*_{\text{met, fin}}$.

$(2)\Rightarrow (1)$ Suppose $\G^*_{\text{met, fin}}$ is not topologically simple, which is witnessed by a closed normal $H\lhd \G^*_{\text{met, fin}}$. Take $g\in \G^*_{\text{met, fin}}\setminus H$ and $\varepsilon>0$ such that \[\inf\left\{\left\|gh^{-1}\right\| : h\in H\right\}>\varepsilon.\] Take a non-trivial $h\in H$ and $0<r<t$ such that $\|g\|,\|h\|\in (r,t)$. Then $C_N(h,G)\B_{\varepsilon}(e)\subseteq H\cdot   \B_{\varepsilon}(e)\not\ni g$, but $g\in \B_t(e)$, contradiction.
\end{proof}

\subsection{Compactness and well-foundedness of trees}\label{sec:tree}

We elaborate in this subsection a bit on the $\varepsilon$-conditions from Theorem \ref{thm:sim_b}(2), making a link with well-founded trees in the descriptive set theory sense \cite[Section 2]{kech}.

Let us work under the notation from Theorem \ref{thm:sim_b}. Fix a family of bounded metric groups $\G$ and a family of metrically internal subsets $\X=\{X_m\}_{m\in\N}$ of $\G^*_{\text{met}}$. 
\begin{definition}
We call a finite sequence $(\varepsilon_0,\varepsilon_1,\ldots,\varepsilon_n)\subset\R_{>0}$ \emph{$(\G,\X)$-small} if (\ref{eq:c}) is not true: $\G^*_{\text{met}} \neq  X_{0} \B_{\varepsilon_0}(e) \cup \ldots \cup X_{N} \B_{\varepsilon_N}(e)$. Otherwise $(\varepsilon_0,\varepsilon_1,\ldots,\varepsilon_n)$ is called \emph{$(\G,\X)$-big}.
\end{definition}

\begin{remark}\label{rem:ext}
\begin{enumerate}
    \item An extension of an $(\G,\X)$-big sequence is $(\G,\X)$-big again.
    
    \item If in a sequence 
\begin{equation}\label{eq*}
(\varepsilon_0,\ldots,\varepsilon_i,\ldots,\varepsilon_j,\ldots,\varepsilon_N)
\end{equation} 
it happens that $\varepsilon_i < \varepsilon_j$, then the sequence
\[(\varepsilon_0,\ldots,\varepsilon_j,\ldots,\varepsilon_j,\ldots,\varepsilon_N)\]
is $(\G,\X)$-big if and only if (\ref{eq*}) is. Hence, we can always restrict ourselves to non-increasing sequences.

    \item The set of all finite small sequences \[T_{\G,\X} = \left\{ \bar{\varepsilon} = (\varepsilon_0,\varepsilon_1,\ldots,\varepsilon_n) \subset \R_{>0} : \bar{\varepsilon}\text{ is }(\G,\X)\text{-small}\right\}\] has a structure of a tree. That is, it is a family of finite sequences such that every initial segment of a sequence in the family also belongs to the family \cite[Section 2]{kech}.
\end{enumerate}
\end{remark}

 An equivalent formulation of Theorem \ref{thm:sim_b} is the following.
\begin{theorem}\label{thm:tree}
$\X$ covers $\G^*_{\text{met}}$ if and only if $T_
{\G,\X}$ has no infinitely long path, i.e. $T_{\G,\X}$ is a well-founded tree \cite[Section 2.E]{kech}.
\end{theorem}

Similar fact can be stated for $\varepsilon$-condition in Theorem \ref{thm:sim2}, where one need to cover balls.

One can assign a rank $\rho(T)$ to a well-founded tree $T$, which is an ordinal. 
We derive some ranks of a well founded $T_
{\G,\X}$ in Example \ref{ex:per} (3).

 \begin{remark}
 Suppose $G^*_{\text{met, fin}}$ is simple (i.e. $(G,\|\cdot\|)$ satisfies the conditions from Theorem \ref{thm:SimpFam}), then the collection \[\rho(G,\|\cdot\|) = \left\{\rho(T_{r,t}) : t>r>0\right\}\] can be regarded as a family of invariants of $(G,\|\cdot\|)$.
\end{remark}

\section{Metric ultraproducts of symmetric groups}\label{sec:sym}

 Let us apply Theorem \ref{thm:SimpFam} and Corollary \ref{cor:stozek} to symmetric groups $S_n$, alternating groups $A_n$ and $S_\infty=\bigcup_{n\in\N} S_n$. 
 
 \begin{definition}\label{def:hamm}
 Let $\|\cdot\|_H$ be the Hamming norm on $S_n$, defined for $\sigma\in S_n$ as \[\|\sigma\|_H = |\supp(\sigma)|,\ \text{ where }\supp(\sigma)=\{i : \sigma(i)\neq i\}.\] 
 \end{definition}
 
 \begin{remark}\label{spr}
 Let us check Assumption \ref{eq:ass} for \[\S=(S_n,\|\cdot\|_H)_{n\in\N\cup\{\infty\}}.\] Observe that any two conjugate elements $\sigma_1$ and $\sigma_1$ from $S_n$ can be conjugated by element with a support contained in $\supp(\sigma_1) \cup \supp(\sigma_2)$. Therefore, a function $F_{\S}(x)=2x$ witnesses that Assumption \ref{eq:ass} holds for $\S$. Observe that $F_\S$ is a linear function, so this assumption with same $F_\S$ also holds for a modified family $\S'=(S_n,c_n\|\cdot\|_H)_{n\in\N\cup\{\infty\}}$, where each $c_n\in\R_{>0}$, which we use in Example \ref{ex:per}.
 \end{remark}
 
 We use a classical result of Brenner from \cite{bren}. One say that $\sigma\in A_n$ is \emph{nonexceptional} if $\sigma^{S_n} = \sigma^{A_n}$. According to \cite[11.1.5, p. 299]{scott}, $\sigma$ is exceptional if and only if all cycles in $\sigma$ have different odd lengths. 
 
 
 \begin{lemma} \label{lem:lulu} Fix natural $n\geq 5$.
 \begin{enumerate}
     \item \cite[3.05]{bren} Let $\sigma\in A_n$ be a nonexceptional permutation with full support $\|\sigma\|_H=n$. If $\ll_r(\sigma)\geq \frac{n-1}{2}$, then $A_n=C_4(\sigma, A_n)$, that is every element of $A_n$ is a product of 4 conjugates of $\sigma$ and $\sigma^{-1}$.
     \item For any $\tau\in S_n$, $\|\tau\|_H\geq 5$ there is a nonexceptional $\sigma\in A_n$ with $\|\tau\sigma^{-1}\|_H\leq 2 + 3=5$ and $\supp(\sigma)=\supp(\tau)$. Moreover  \[S_n = C_{16 + 4\frac{n}{\|\sigma\|_H}}(\sigma,S_n) B_2(e).\]
 \end{enumerate}
 \end{lemma}
 \begin{proof} 
 (2) We may assume that $\tau\in A_n$ (by multiplying by a transposition).   Suppose $\tau$ is exceptional.  Define a cycle $\rho(m) = (1,2,\ldots,m-1,m)$ and \[\pi(m) = (1,2,\ldots,m-4)\circ (m-3,m-2)\circ (m-1,m),\] for $m\in\N_{\geq 5}$. Observe that \[\rho(m) \pi(m)^{-1} = (m-4,m-2,m).\]
 Clearly $\rho(2m+1),\pi(2m+1)\in A_{2m+1}$. Since $\tau$ is exceptional, it can be written as \[\tau = \tau' \circ \rho(2m+1)\] for some $m\geq 5$ (where $\tau'$ fixes pointwise $\{1,2,\ldots,m-1,m\}$). Then $\sigma = \tau'\circ \pi(2m+1)$ is nonexceptional (as it has two 2-orbits) with $\|\tau\sigma^{-1}\|_H\leq 3$. To prove the last part, it is enough to argue that $A_n=C_{16+\frac{4n}{\|\sigma\|_H}}(\sigma,S_n)$, as $\B_2(e)$ contains all transpositions. Our aim is to find $\sigma_\infty\in C_{4+\frac{n}{\|\sigma\|_H}}(\sigma,S_n)$ of the full support $\supp(\sigma_\infty)=\{1,\ldots,n\}$. Then (1) and (2) gives the conclusion. 
 
 Let us construct such $\sigma_\infty$ (using the ideas from \cite[Lemma 2.5]{elek} and \cite{thomas}). Let $X_1=\supp(\sigma)\subseteq \{1,2,\ldots,n\}$ and $k=|X|=\|\sigma\|_H$. There is a partition \[\{1,2,\ldots,n\} = X_1\cup\ldots\cup X_{\lfloor\frac{n}{k} \rfloor}\cup Y,\] where $|X_i|=|X|=k$ and $|Y|=n-k\cdot\lfloor\frac{n}{k} \rfloor<k$.
  Clearly $\sigma\in \alt(X)=A_k$ satisfies (2), that is $\alt(X)=C_4(\sigma, \alt(X))$. Hence, there are $\rho_i\in S_n$ such that $\sigma_i:= \sigma^{\rho_i}\in \alt(X_i)$ and $\supp(\sigma_i) = X_i$. There is also $\sigma_0\in C_4(\sigma,S_n)$ such that $\supp(\sigma_0) = Y$. Consider \[\sigma_\infty = \sigma_1 \sigma_2\ldots \sigma_{ \lfloor\frac{n}{k} \rfloor } \sigma_0.\] Clearly $\supp(\sigma_\infty)=\{1,\ldots,n\}$ and $\sigma_\infty\in C_{4+ \lfloor\frac{n}{k} \rfloor }(\sigma,S_n)$.
 \end{proof}
 
 \begin{example} \label{ex:per}
 \begin{enumerate}
    \item Let $\G=\left(S_n, \frac{1}{n}\|\cdot \|_H\right)_{n \in \N}$. Then $\G^*_{\text{met, fin}}$ is called a \emph{universal sofic group} \cite{pestov1}. It is known that ${\G}^*_{\text{met, fin}}$ is simple group \cite[Proposition 2.3(5)]{elek}.  In fact it is boudedly simple (see Definition \ref{def:bsim}, Theorem \ref{cor:bbsim}). This can be proved as in (2) below.
    
    \item Fix a sequence $\bar{c}=\left(c_n\right)_{n\in\N}$ of positive real numbers such that $\lim_{n\to\U}c_n=0$. Consider \[\G=\left(S_\infty, c_n\|\cdot \|_H\right)_{n \in \N},\]
    where $S_\infty=\bigcup_{n\in\N} S_n$. Then $\G^*_{\text{met, fin}}$ is a simple group by Corollary \ref{cor:stozek} and Lemma \ref{lem:lulu}. Indeed, let $t>r>0$ and let $\bar{\varepsilon} = (\varepsilon_0, \varepsilon_1,\ldots)\subset \R_{>0}$. Define $N= 16+4\frac{t}{r}$ and take arbitrary $\varepsilon >0$. Suppose $n\in\N$ is such that $c_n<\frac{\varepsilon_N}{2}$ and $c_n<\frac{\varepsilon}{5}$. There are $\U$-many such $n$. Take $\tau\in S_\infty$ with $\|\tau\|_H \in \left(\frac{r}{c_n},\frac{t}{c_n} \right]$. There is a  nonexceptional $\sigma\in A_n$ (by (\ref{lem:lulu}(2)) with $\left\|\tau\sigma^{-1}\right\|_H\leq 5<\frac{\varepsilon}{c_n}$ and 
     \[\B_{\frac{t}{c_n}}(e) \subseteq  C_N\left(\sigma,S_\infty\right) \B_{\frac{\varepsilon_N}{c_n}}(e).\] Hence $\G^*_{\text{met, fin}}$ is  simple by Corollary \ref{cor:stozek}.
     %
    %
    
    \item Ranks of a well founded $T_
{\G,\X}$ in this Example are $\leq 16+4\frac{t}{r}$.
\end{enumerate}
\end{example}


\begin{example}\label{ex:cover}  An equality (\ref{eq:22}) from Theorem \ref{thm:ccomp} cannot be simplified to apparently simpler condition $\G^*_{\text{met}} = X_1 \cup\ldots\cup X_N$. 
Indeed, consider $\G=\left(S_n, \frac{1}{n}\|\cdot \|_H\right)_{n \in \N}$ and consider subsets from Example \ref{ex:internal} (\ref{balls}):
\begin{align*}
X_0 = \frac{\prod_{n\in\N} \left\{\sigma \in S_n : \frac{1}{n}\|\sigma\|_H=1 \right\}}{\E},\ \ 
X_m = \frac{\prod_{n\in\N} \B_{1-\frac{1}{m}}(e,G_n)}{\E}, \text{ for } m>0.
\end{align*}
Of course, $X_0 \cup\ldots\cup X_N= X_0 \cup X_N \neq \G^*_{\text{met}}$, for any natural number $N$. Observe that $\G^*_{\text{met}} =\bigcup_{n\in\N}X_n$ as \[ \left\{  (\sigma_n) \in \prod_{n\in\N}S_n : \lim_{n\to\U}\frac{1}{n}\|\sigma_n\|_H = 1 \right\} = \E \cdot \left\{  (\sigma_n) \in \prod_{n\in\N}S_n : \forall\ n\in\N\ \  \frac{1}{n}\|\sigma_n\|_H = 1 \right\},\] where $\E$ is from (\ref{inf}).  
 \end{example}

\begin{question}
Example \ref{ex:per} provides many ultraproducts of groups which are simple. The natural question is: are they pairwise isomorphic for a fixed ultrafilter $\U$? 
To be more precise, for $\bar{c}=(c_n)_{n\in\N} \subset \R_{>0}$ define $\S(\bar{c})=\left(S_\infty, c_n\|\cdot \|_H\right)_{n \in \N}$ and consider metric ultraproduct $\S(\bar{c})^*_{\text{met, fin}}$. Observe that if $\lim_{n\to\U}\frac{c_n}{d_n} \neq 0, +\infty$, then $\S\left(\bar{c}\right)^*_{\text{met, fin}}= \S\left(\bar{d}\right)^*_{\text{met, fin}}$. It is also true in general? It would be interesting to determine for example whether  \[\left(S_\infty, \frac{1}{n}\|\cdot \|_H\right)^*_{\text{met, fin}} \cong \left(S_\infty, \frac{1}{n^2}\|\cdot \|_H\right)^*_{\text{met, fin}}\] holds for any non-principal ultrafilter? We suspect also that each $\S(\bar{c})^*_{\text{met, fin}}$ is a universal sofic group \cite{pestov1}.
\end{question}

\begin{example}\label{ex:dos}
Consider again $\S = \left(S_n, \frac{1}{n}\|\cdot\|_H\right)_{n\in \N}$. Then metric ultraproduct $\S^*_{\text{met}}$ is a perfect group (as it is simple), however discrete ulraproduct $\S^*$ is not perfect, as for any positive $n \in \N$ a group $S_n$ is not perfect (as $[S_n,S_n]=A_n)$.
\end{example}

\section{Simple groups via approximation} \label{sec:app}

Theorem \ref{thm:iet} below yields more simple ultrapowers. Intuitively it says the following: if a metric group $H$ is approximated (in a certain sense) by a family $\G=(G_n)_{n\in\N}$ of metric groups and metric ultraproduct $\G^*_{\text{met, fin}}$ is simple, then metric ultrapower $H^*_{\text{met, fin}}$ of $H$ is also simple. 
 
We prove in Theorem \ref{thm:iett}, as a corollary of \ref{thm:iet}, that  $\IET$ group (the group of all interval exchange transformations of $[0,1]$) with a natural metric is metrically simple (Definition \ref{def:metsim}). 


\begin{theorem} \label{thm:iet}
Let $\G=\left(G_n, \|\cdot\|'_n\right)_{n \in \N}$ and $\H=(H, \|\cdot\|)$ be metric groups and fix a nonprincipal ultafilter $\U$. Suppose that $\G^*_{\text{met, fin}}$ is a simple group (boundedly simple resp.). Assume that the following condition is true for $\H$: 
\begin{quote}
for any $t>0$ and any $\varepsilon >0$ and any $h_1,h_2\in H$ with $\|h_1\|,\|h_2\|\leq t$ there are $\U$-many $n\in\N$ such that there are $g_1,g_2\in G_n$ and $\phi_{n}\colon G_n \rightarrow H$ an isometric homomorphism (i.e. $\left\|\phi_{n}(g)\right\| = \|g\|'_n$) satisfying $\left\|\phi_n(g_i){h_i}^{-1}\right\| <\varepsilon$, for $i=1,2$.
\end{quote}
Then $H^*_{\text{met, fin}}$ is a simple group (boundedly simple resp.) with respect to $\U$.
\end{theorem}

\begin{proof} Let us prove $(2)$ of 
Remark \ref{covid} for $\H$. Fix a positive numbers $t>r>0$ and an infinite sequence of positive reals $(\varepsilon_0, \varepsilon_1,\ldots)$. 

Since $\G^*_{\text{met, fin}}$ is a simple group, Theorem  \ref{thm:SimpFam} (2) applied to $\G$ and to $\left(\frac{\varepsilon_0}{2}, \frac{\varepsilon_1}{2},\ldots\right)$ gives $N\in\N$ such that
\begin{equation}\label{eq:i0}
I=\left\{n \in \N: \left(\frac{\varepsilon_0}{2}, \frac{\varepsilon_1}{2},\ldots,\frac{\varepsilon_N}{2}\right)\text{ is }\left(\frac{r}{2},2t\right)\text{-big for }G_n \right\} \in \U,
\end{equation} see Definition \ref{def:rt} for 'big' notation. 
Suppose $h_1\in H$ with $\|h_1\|\ > r$ . Take any $h_2\in B_{t}(e)$ in $H$. We have to prove that $h_2\in C_j(h',H) B_{\varepsilon_j}(e)$ for some $0\leq j \leq N$ and $h'\in B_{\varepsilon_j}(h_1)$. Define \[\varepsilon'= \min\left\{\frac{r}{2},t,\frac{\varepsilon_0}{2}, \frac{\varepsilon_1}{2 },\ldots,\frac{\varepsilon_N}{2} \right\}>0\]
and take $n \in I$ suitable for $\varepsilon'>0$, $h_1,h_2\in H$, as in the assumption of the theorem. There exist $g_1,g_2 \in G_n$ and $\phi_n\colon G_n\to H$ such that $\left\|\phi_n(g_i)h_i^{-1}\right\|<~\varepsilon'$, $i=1,2$. Then \[\left\|g_1\right\|'_n = \left\|\phi_{n}(g_1)\right\| \in \Big[\|h_1\| - \left\|h_1^{-1}\phi_{n}(g_1)\right\|, \|h_1\| + \left\|h_1^{-1}\phi_{n}(g_1)\right\|\Big]\subseteq \left(\frac{r}{2},2t\right).\] Similarly $\|g_2\|'_n < 2t$. Therefore by (\ref{eq:i0}) (since $n\in I$) 
 \begin{equation} \label{eq:g}
 g_2 \in \B_{2t}(e)\subseteq \bigcup_{j=0}^N C_j(g_1,G_n) \B_{\frac{\varepsilon_j}{2}}(e).
 \end{equation}
Applying $\phi_{n}$ to (\ref{eq:g}) gives that 
$\phi_{n}(g_2)\in \bigcup_{j=0}^N C_j(\phi_{n}(g_1),H) \B_{\frac{\varepsilon_j}{2}}(e)$.
Since $\left\|\phi_{n}(g_2)h_2^{-1}\right\|< \varepsilon'\leq \frac{\varepsilon_j}{2}$, we conclude that $h_2=h_2 \phi_{n}(g_2)^{-1} \phi_{n}(g_2)\in \B_{\varepsilon'}(e) \bigcup_{j=0}^{N} C_j(\phi_{n}(g_1),H) \B_{\frac{\varepsilon_j}{2}}(e) \subseteq \bigcup_{j=0}^{N} C_j(\phi_{n}(g_1),H) \B_{\varepsilon_j}(e)$, which finishes the proof (for $h' = \phi_{n}(g_1)$).
\end{proof}

\section{$\IET$ group} \label{sec:iet}

Let us apply Theorem \ref{thm:iet} to get interesting examples of simple metric ultraproducts.

An \emph{interval exchange transformation} is a  bijective map $f\colon [0,1] \rightarrow [0,1]$ which is piecewise translation,  continuous on the right with finitely many discontinuity points. Note that any such transformation is given by a pair: 
a sequence of real numbers 
$0=a_1<a_2<\ldots<a_n=1$ and a permutation $\sigma \in S_n$. Denote such transformation by $T=T(a_1,a_2,\ldots,a_n, \sigma)$. Such $T$ translates $\left[a_i,a_{i+1}\right)$ onto $\left[a_{\sigma(i)},a_{\sigma(i+1)}\right)$.

The set of all interval exchange transformation with composition form a group, which we denote by $\IET$. A bi-invariant norm of an element $g \in \IET$ is a Lebesgue measure of its support:
\[\|g \|_\mu= \mu(\supp(g)).\]

\begin{theorem} \label{thm:iett}
Any metric ultrapower of $\IET$ with respect to $\|\cdot\|_\mu$  is a simple group, in fact boundedly simple (Definition \ref{def:bsim}). 
\end{theorem}

\begin{proof}
Consider $\G=\left( S_n, \frac{1}{n}\|\cdot\|_H \right)_{n \in \N}$, where $\|\cdot\|_H$ is the Hamming norm and $\H=(\IET, \|\cdot~\|_\mu)$. Example \ref{ex:per} implies that $\G^*_{\text{met, fin}}$ is a simple group (in fact $\G^*_{\text{met, fin}}$ is boundedly simple). It is enough to show, by Theorem \ref{thm:iet}, that for any $\varepsilon >0$ there is a natural number $N$ such that for any $n \geq N$ and any $h_1, h_2 \in \IET$ there are $\sigma_1, \sigma_2 \in S_n$ and  an isometric homomorphism $\phi \colon S_n \to \IET$ such that $\left\|\phi(\sigma_i)h_i^{-1}\right\| <\varepsilon$, $i=1,2$.

Let $h_1, h_2 \in \IET$ where 
$h_1= T(a_1,a_2,\ldots,a_k, \sigma_1)$, and $h_2= T(b_1,b_2,\ldots,b_l, \sigma_2)$.
Let $n \in \N$ be such that 
$n>\frac{k}{\varepsilon}$, $n>\frac{l}{\varepsilon}$  and
$\frac{1}{n} < \min\left( \{a_{i+1}-a_i: i=1,2,\ldots,k \} \cup \{b_{i+1}-b_i: i=1,2,\ldots,l \}\right)$. Let $a'_i =\frac{\lfloor n a_i \rfloor }{n}$ and $b'_j =\frac{\lfloor n b_j \rfloor }{n}$ for $1\leq i \leq k$, $1\leq j \leq l$. Define $h'_1= T(a'_1,a'_2,\ldots,a'_k, \sigma_1)$, and $h'_2= T(b'_1,b'_2,\ldots,b'_l, \sigma_2)$.
Elements $h'_1, h'_2$ are given also by a sequence $\left(\frac{1}{n}, \frac{2}{n},\ldots\frac{n}{n}\right)$ and permutations $\sigma'_1, \sigma'_2 \in S_n$:
\[h'_1= T\left(\frac{1}{n}, \frac{2}{n},\ldots\frac{n}{n}, \sigma'_1 \right),\text{ and }h'_2= T\left(\frac{1}{n}, \frac{2}{n},\ldots\frac{n}{n}, \sigma'_2\right).\]
Define a monomorphism $\phi\colon S_n \rightarrow \IET$ by
$\phi( \delta)= T\left(\frac{1}{n}, \frac{2}{n},\ldots\frac{n}{n}, \delta\right)$. Thus $h'_i=\phi\left(\sigma'_i\right)$. We have \[\left\|\phi\left(\sigma'_i\right)h_i^{-1}\right\|_\mu<\max\left(\frac{k}{n},\frac{l}{n}\right) < \varepsilon.\] It is easy to check that  $\phi$ is an isometry.
\end{proof}

\section{Metrically uniformly simple groups} \label{sec:met}


Our examples of simple metric ultraproducts from sections \ref{sec:sym}, \ref{sec:iet} do satisfy another property, which we call \emph{metric uniformly simplicity} in Definition \ref{msf} below. 
%
%
%
Recall from Definition \ref{def:bsim} that $C_N(g,G) := \left(g^{G} \cup g^{-1 G}\right)^{\leq N}$.  

\begin{definition} \label{msf} \mbox{}
\begin{enumerate}
    \item A group with pseudo-metric $(G, \|\cdot\|)$ is called 
\emph{metrically uniformly simple} if for all $t>r>0$, there is $N \in \N$ such that $C_{N}(g,G)\supseteq \B_t(e)$ holds for all $g \in G$ with $\|g\| \in (r,t]$.
    \item A family of groups with pseudo-metrics  $\G=(G_n, \|\cdot\|_n)_{n\in \N}$ is called \emph{metrically uniformly simple} if for all $t>r>0$, there is $N \in \N$ such that $C_{N}(g,G_i)\supseteq \B_t(e)$ holds for all $n\in \N$ and for any $g \in G_n$ with $\|g\|  \in (r,t]$.
\end{enumerate}
\end{definition}  

We conjecture the following.

\begin{conjecture} \label{con:gw}
If a metric ultraproduct $(\G^*_{\text{met,fin}},\|\cdot\|)$ of a family of pseudo-metric group $\G$ is simple, then $\G^*_{\text{met,fin}}$ must be metrically uniformly simple (Definition \ref{msf}(1)). 
\end{conjecture}

\begin{remark} \label{MetPrdMUS}
Observe that metric uniformly simplicity is preserved under taking metric ultraproducts, that is metric ultraproduct of metrically uniformly simple family of groups is metrically uniformly simple.
\end{remark}

\begin{lemma} 
Let $\G=(G_n, \|\cdot\|_n)_{n \in \N}$ be a family of metric groups. A metric ultraproduct $G^*_{\text{met, fin}}$ is metrically uniformly simple if and only if for any $t>r>0$ there is $N\in\N$ such that for any $\varepsilon >0$ the following holds  \[\text{for }\U\text{-almost all }n\in\N\ \   C_N(g,G_n)\cap \B_t(e)\text{ is }\varepsilon\text{-dense in }\B_t(e)\text{ in }G_n.\]
\end{lemma}
\begin{proof} $(\Rightarrow)$ Suppose $G^*_{\text{met, fin}}$ is metrically uniformly simple. Let $e \neq\bar{g} \in G^*_{\text{met, fin}}$
and let $(g_n)_{n \in \N} \in \prod_{n\in\N} G_n$ be such that $\bar{g} = (g_n)_{n \in \N}\E$. Fix $0< r< \| \bar{g} \| \leq t$ and let $N$ be as in Definition \ref{msf}. We have that  \[\left\{n\in\N : C_N(g_n,G_n)\cap \B_t(e)\text{ is $\varepsilon$-dense in }\B_t(e)\text{ in }G_n\right\} \in \U,\]
for any $\varepsilon >0$.
So we conclude that $C_{N}(\bar{g},G^*_{\text{met, fin}}) \supseteq  \B_t(e)$ in $G^*_{\text{met, fin}}$.

$(\Leftarrow)$ For contradiction, suppose that there are $t>r>0$ such that for every natural number $N$ there is $\varepsilon_N >0$ such that $U_N\in\U$, where $U_N=$
\[\left\{n\in\N : (\exists g_{N,n} \in G_n)\|g_{N,n}\|_n \in (r,t] \land C_N(g_{N,n},G_n)\cap \B_t(e)\text{ is not $\varepsilon_N$-dense in }\B_t(e) \right\}.\] 
For $N \in \N$ and $n \in U_N$ let $h_{N,n} \in \B_t(e)$ such that
\[ \inf \left\{ \|h_{N,n}g^{-1} \|: g \in  C_N(g_{N,n},G_n)  \right\} > \varepsilon_N. \]
This gives in ultraproduct $\G^*_{\text{met, fin}}$ elements
\[\overline{g_N}= (g_{N,1},g_{N,2}, \ldots)\E\text{ and } \overline{h_N}= (h_{N,1},h_{N,2}, \ldots)\E,\] such that $\|\overline{g_N} \| \in (r,t]$, $\|\overline{h_N} \| \leq t$ and 
\[ \overline{h_N} \not\in C_N\left(\overline{g_N},G^*_{\text{met, fin}}\right),\] for any $N\in\N$. Hence, $G^*_{\text{met, fin}}$ can not be metrically uniformly simple.
\end{proof}


From now we assume that all norms are bounded, that is $\|\cdot\|\leq 1$.

%
We will show that Conjecture \ref{con:gw} is true under some natural assumption. First, in Theorem \ref{Us->mus} we prove that any simple metric ultraproduct of a family with $(\star)$-property is metrically uniformly simple. Next, in Theorem \ref{2Us->mus} we show that Conjecture is true if underlying metric ultraproduct has $(\star)$-property.


Let us introduce some notion.

\begin{definition}
For a group $G$ and $g \in G$ denote 
$N(g,G)= \min \{n \in \N: C_n(g,G) = G\}$.
\end{definition}

Let us introduce $(\star)$-property, which is crucial in arguments below.

\begin{definition}\label{def:star}
We say that a family of groups $\G$ has $(\star)$-property if (1) and (2) are true:
\begin{enumerate}
\item
there is a natural number $N$ such that for any $G \in  \G$ there is $g \in G$  such that $C_{N}(g,G)=G$;
%
%
\item
for any $k\in\N$ there is $l\in\N$ such that for any group $G \in \G$ and elements $g,h \in G$ if $N(g,G), N(h,G)\geq l$, then $N(gh,G)\geq k$.
\end{enumerate}
\end{definition}

Now we prove that every family of metric groups which is metrically uniformly simple has $(\star)$-property.


\begin{fact} \label{**prop}
If $\G$ is a family of metric groups which is 
metrically uniformly simple (see Definition \ref{msf}), then $\G$ has $(\star)$-property.
\end{fact}
\begin{proof}
Let $\G=(G_n, \|\cdot\|_n)_{n\in \N}$ be a metrically uniformly simple family of groups.
For a number $k \in \N$ take $l \in \N$ such that for any $i \in \N$ and any $g \in G_i$ with $\|g\| > \frac{1}{2k}$ we have $C_l(g, G_i)=G_i$.

So if we take any $k \in \N$ and any $g,h \in G_i$
with $N(g, G_i), N(h, G_i) >l$ then $\|g\|, \|h\| <\frac{1}{2k}$ so $\|gh\| < \frac{1}{k}$, and so $N(gh, G_i)>k$.
\end{proof}

\begin{remark}
By theorem 1.1 \cite{liesha}, a family of finite simple groups has a $(\star)$-property. We see as consequence of Theorem \ref{PslC}, that $(\PSL_n(\CC))_{n\in\N}$ has  $(\star)$-property.
\end{remark}

Next notion allows us to define a set of infinitesimal sequences in products of any (even in non-necessarily metric) groups.

\begin{definition}\label{def:zz}
For a family of groups $\G=(G_n)_{n\in N}$ and a nonprincipal ultrafilter $\U$ on $N$ define 
\[ Z_{\U}= \left\{ (g_n)_{n\in \N}: \forall k \in \N \ \  \{n \in \N: N(g_n, G_n)>k \} \in \U \right\}.\]
\end{definition}


\begin{fact} \label{ll}
If a family $\G = (G_n)_{n\in \N}$ has a $(\star)$-property, then $Z_{\U}$ is a normal subgroup of $\prod_{n\in \N} G_n$, for any non-principal ultrafilter $\U$. Moreover $\frac{\prod_{n\in \N}G_n}{Z_{\U}}$ is a non-trivial boundedly simple group (see Definition \ref{def:bsim}).
\end{fact}

\begin{proof}
Obviously $Z_{\U}$ is a normal subset of  $\prod_{n\in \N} G_n$, that is $Z_{\U}$ is closed under conjugation. It is enough to prove that $Z_{\U}$ is closed under multiplication. Let $g=(g_n)_{n\in N},  h=(h_n)_{n\in \N} \in Z_{\U}$.
Fix $k \in \N$.
Observe that \[\{n \in \N: N(g_n \cdot h_n, G_n)>k \} \in \U,\] 
so $(g_n \cdot h_n)_{n\in \N} \in Z_{\U}$.
Take a natural number $l$ as in definition of $(\star)$-property.
Put \[U_g=\{n\in  \N:  N(g_n, G_n)>l  \}\text{ and } U_h=\{n\in  \N:   N(h_n, G_n)> l  \}.\]
Let $U= U_g \cap U_h \in \U$. Then $N(g_n, G_n), N(h_n, G_n)\geq l$, for any $n \in U$. Thus we have that $N(g_n \cdot h_n, G_n) \geq k$, hence $(g_n)_{n\in N}\cdot (h_n)_{n\in \N} \in Z_{\U}$.

Let us prove the moreover part. By $(1)$ of $(\star)$-property  $ Z_{\U} \neq \prod_{n\in \N} G_n$.
Let  $ (g_n)_{n\in \N} \notin Z_{\U} $.
There is a natural number $k$ such that  
$\{n \in \N: N(g_n, G_n) \leq k \} \in \U$, so
\[\{n \in \N: C_k(g_n, G_n)=G_n \} \in \U.\] 
Finally, 
\[C_k\left((g_n)_{n\in \N} Z_{\U}, \frac{\prod_{n\in \N}G_n}{Z_{\U}}\right)=
 \frac{\prod_{n\in \N} C_k(g_n, G_n)}{Z_{\U}}= \prod_{n\in \N}\frac{G_n}{Z_{\U}},\] so $\frac{\prod_{n\in \N}G_n}{Z_{\U}}$ is boundedly simple.
\end{proof}
%


\begin{fact} \label{E<Z}
Let $\G=(G_n, \|\cdot\|_n)_{n \in \N}$ be a family of pseudo-metric groups having $(\star)$-property. Then $\E \leq Z_{\U}$. 
Moreover if  $G^*_{\text{met}}$ is simple then  $\E= Z_{\U}$ (see Definition (\ref{inf}) for $\E$). 
\end{fact}

\begin{proof}
Let $\bar{g}=(g_n)_{n \in \N} \in \E$ and fix a natural number  $k>0$.
Take $\bar{h}=(h_n)_{n \in \N} \in \Pi_{n \in \N} G_n $ with $\|\bar{h}\|=r>0$.
We have that 
\[ \left\{n \in \N: \|g_n\|< \frac{r}{k} \right\} \in \U \]
so
\[ \left\{n \in \N: N(g_n, G_n) = \min \{n \in \N: C_n(g_n,G_n) = G_n\} >k\right\} \in \U. \]
Finally,  $\bar{g} \in Z_{\U}$. Moreover part is clear.
\end{proof}

Any known to us simple metric ultraproduct is in fact a metric ultraproduct of family with $(\star)$-property. For example, $\S = \left(S_n, \frac{1}{n}\|\cdot\|_H\right)_{n\in \N}$ has no $(\star)$-property (as (2) fails), but its metric ultraproduct $\S^*_{\text{met}}$ equals to metric ultraproduct of $\A = \left(A_n, \frac{1}{n}\|\cdot\|_H\right)_{n\in \N}$ which has $(\star)$-property. 

We will prove below that any simple metric ultraproduct of $(\star)$-property family is metrically uniformly simple.

\begin{theorem} \label{Us->mus}
If a family of metric groups has a $(\star)$-property, then any its simple metric ultraproduct is uniformly metrically simple.
\end{theorem}

\begin{proof}
Let $\G=(G_i, \|\cdot\|_i)_{i \in \N}$ be a family of metric groups with $(\star)$-property such that $G^*_{\text{met}}=\Pi_{i\in \N} G_i/\E$ is simple.
Suppose that $\G^*_{\text{met}}$ is not uniformly metrically simple.
Then, there are a number $r>0$ and a sequence $(\bar{g}_n=(g_{n,1},g_{n,2},g_{n,k}, \ldots))_{n \in \N}$ such that $\|\bar{g}_n\|>r$ and $C_n(\bar{g}_n, \G^*_{\text{met}}) \neq \G^*_{\text{met}}$  for any $n \in \N$.
So, for any natural number $n$ there is a set $U_n \in \U$ such that $\|g_{n,i}\|>r$ and $C_n(g_{n,i}, G_i) \neq G_i$ for any $i \in U_n$. We can assume that
\[\{1,2,\ldots, m\}\cap U_n=\emptyset\]
for any $m<n$. For any number $i \in \N$ let $k_i$ be maximal number such that $i \in U_{k_i}$ and let $h_i=g_{k_i,i}$. Define, $\bar{h}=(h_i)_{i\in \N}$. We have that $\bar{h} \in Z_{\U}$, because for any number $n$ a set $\{i: C_n(h_i, G_i) \neq G_i \} \supseteq U_n$ is in ultrafilter.
A group  $\G^*_{\text{met}}$ is simple, so by Fact \ref{E<Z} we have, $\E=Z_{\U}$.
So, $\bar{h} \in \E$, which gives a contradiction with $\|h_i\|>r>0$ for any $i$.
\end{proof}


\begin{theorem} \label{2Us->mus}
If a metric ultraproduct $\G^*_{\text{met}}$ is simple and has a $(\star)$-property, then $\G^*_{\text{met}}$ is metrically uniformly simple (Definition \ref{msf}).
\end{theorem}

\begin{proof}
Fix  $\G = (G_n,\|\cdot\|_n)_{n\in \N}$ is a family of metric groups and $\U$ an ultrafilter on $\N$. Let $g_1, g_2, g_3, \ldots \in \G^*_{\text{met}}$, such that 
$C_n(g_n, \G^*_{\text{met}}) \neq \G^*_{\text{met}}$.
Since $(\star)$-property holds in $\G^*_{\text{met}}$, we can define  $Z_{\U}$, a normal subgroup of $(\G^*_{\text{met}})^{\omega}$ and see that 
$\bar{g}=(g_1,g_2,\ldots) \in Z_{\U}$.
Since $\G^*_{\text{met}}$ is simple,  $(G^{*}_{\text{met}})^*_{\text{met}}=(\G^*_{\text{met}})^{\omega}/\E$ is also  simple.
Hence $\E$ is a maximal normal subgroup of $(\G^*_{\text{met}})^{\omega}$ and $Z_{\U} = \E$.
Since $\bar{g} \in \E$ then for any $r>0$ there is a natural number $i$ such that $\|g_i \| <r$.
So a group $\G^*_{\text{met}}$ must be uniformly metrically simple.
\end{proof}



\begin{corollary}\label{Us->bound}
Suppose for each $n\in\N$, a group $G_n$ is equipped with two norms $\|\cdot\|_n$ and $\|\cdot\|'_n$. Let $\G=(G_n, \|\cdot\|_n)_{n \in \N}$ and $\G'=(G_n, \|\cdot\|'_n)_{n \in \N}$  and let $\U$ be a non-principal ultrafilter on $\N$. If $\G$ is metrically uniformly simple (see Definition \ref{msf}) and $\G^{'*}_{\text{met}}$ is simple, then $\G^{'*}_{\text{met}}$ is   metrically uniformly simple.
\end{corollary}

\begin{proof}
By Fact \ref{**prop} metrically uniformly simple family $\G=(G_n, \|\cdot\|_n)_{n \in \N}$ has $(\star)$-property. It is enough to use Theorem \ref{Us->mus}.
\end{proof}

Let us finish this section with some applications.

Our $(\star)$-property gives a king of rigidity, in the sense of the Corollary \ref{cor:lulu} below, which illustrated by Example \ref{sztyw}. 

\begin{corollary}\label{cor:lulu}
Fix a group $G$ with $(\star)$-property. If $G$ is not simple, then every metric ultrapower $G^*_{\text{met}}$ of $G$ is not simple (that is, for any bounded choice of metric on $G$).
\end{corollary}

\begin{proof}
Let $G$ be non-simple group with $(\star)$-property and for contradiction suppose that there is a norm $\|\cdot\|$ on $G$ such that some metric ultrapower $\G^*_{\text{met}}$ is simple.
Take non-identity element $g\in G$ such that we have $C_N(g,G)\neq G$ for any number $
N\in \N$. Observe that a sequence $\bar{g}=(g,g,\ldots)$ is an element of subgrop $Z_{U}$. 
By Fact \ref{E<Z} we have $\bar{g}\in Z_{\U}= \E$, so $\|g\|=0$ which contradicts with $g$ is not identity.
\end{proof}

\begin{example} \label{sztyw}
Consider $\A = \left(A_n, \frac{1}{n}\|\cdot\|_H\right)_{n\in \N}$ and a pseudo norm $\| (g_i)_{i \in \N} \|= \lim_{i \rightarrow \U} \|g_i\|_H$
on algebraic ultraproduct $\A^*=\prod_{i\in I}A_n/\U$ of $\A$. By Lemma \ref{lem:lulu}(2) $C_{N}(\bar{g},\A^*)=\A^*$ holds for any element $\bar{g} \in \A^*$ of positive norm, where $N=16+\frac{4}{\|\bar{g}\|}$. Since $\A^*$ has $(\star)$-property by Remark \ref{**prop}, by Corollary \ref{cor:lulu}, there is no bi-invariant norm on $\A^*$, such that
its metric ultrapower is simple. 
\end{example}
%
%

A similar result as in \cite{nikol} gives full description of set of maximal normal subgroups of product family with $(\star)$-property.

\begin{proposition}\label{lll}
Let $\G=(G_n)_{n\in \N}$ be a family with $(\star)$-property and
suppose that $H$ is a maximal normal subgroup of  $\prod_{n\in \N} G_n$.
Then there exists $\U$, an ultrafilter on $\N$ such that $H= Z_{\U}$.
\end{proposition}

\begin{proof}
Suppose that  $\prod_{n\in \N} G_n/H$ is a simple group.
For $\bar{h} \in H$ and $k \in \N$ let
\[ A(\bar{h},k)=\{ n \in \N: N(h_n, G_n)>k \}.\]
Let $\U_0=\{A(\bar{h},k): \bar{h} \in H, k \in \N \}$.
\begin{claim}
$\U_0$ has finite intersection property, that is, any finite subfamily of  $\U_0$ has a non-empty  intersection.
\end{claim}
If not, suppose that for some $\bar{h}_1, \bar{h}_2,\ldots, \bar{h}_t \in H$ and 
$k_1, k_2,\ldots, k_t \in \N$ we have 
$A(\bar{h_1},k_1) \cap A(\bar{h_2},k_2) \cap\ldots \cap A(\bar{h_t},k_t)=\emptyset$.

Let $k= \max \{k_1, k_2,\ldots, k_t \}$, we have that
$A(\bar{h_1},k) \cap A(\bar{h_2},k) \cap\ldots \cap A(\bar{h_t},k)=\emptyset$. Let $\bar{h_i}=(h_{i,j})_{j \in \N}$, for $i=1,2,\ldots,t$.

We see that for any $j \in \N$ there is $i=1,2,\ldots,t$ such that $N(h_{i,j}, G_i) \leq k$,
so $C_k(h_{i,j}, G_i)= G_i$.
Since conjugacy classes of elements  $\bar{h}_1, \bar{h}_2,\ldots, \bar{h}_t \in H$ generate  $\prod_{n\in \N} G_n$, we have a contradiction and claim is proved.

Let $\U$ be an ultrafilter on $\N$ extending $\U_0$.
We see that any element of $H$ belongs to $Z_{\U}$.
Since $H$ and  $Z_{\U}$ are maximal subgroups, we have $H=Z_{\U}$.
\end{proof}

\section{Metric ultrapowers of linear groups} \label{sec:lin}

We give more applications of our theorems, by using a model-theoretic argument and results of Liebeck-Shalev from \cite{liesha}. 

Let us recall an important bi-invariant pseudo-norm on linear groups from \cite{nikol, thoms}. By $F$ we always denote a field. Let $\F_q$ denote the finite field of order $q$.

The \emph{Jordan length} $\ll_J$ of $A\in \GL_n(F)$ is defined as:
\begin{equation}\label{jordan}
\ll_{J}(A)= \frac{1}{n} \cdot \min_{\lambda \in F^{*}} \rk(A- \lambda I_n),
\end{equation}
where $\rk(M)$ is the rank of a  matrix $M$ and $I_n$ is the $n\times n$ identity matrix. The Jordan length is a pseudo norm on $\GL_n(F)$ \cite[p, 79]{thoms}.
   

We use the following deep fact, which can be derived from \cite[Lemma 4.1]{liesha}, see also \cite[Lemma 6]{thmsch}.

\begin{lemma}{\cite[Theorem 1.1]{liesha}} \label{lem:cc}
 There is a constant $C\in\N_{>0}$ such that for any $n\in\N$ and any finite field $\F$ the following is true for $H=\SL_n(\F)$ and $N\in\N$:
 \begin{equation} \label{eq:ls}
 \text{ for any }A \in H\setminus Z(H)\text{, if }\ll_J(A)\cdot N \geq C\text{, then }C_N(A, H)=H.
 \end{equation}
\end{lemma}

The crucial point in our argument is that, we can express the conclusion (\ref{eq:ls}) of Lemma \ref{lem:cc} as a first order sentence of a field in the language of rings $L=\{+,\cdot,0,1\}$, for fixed $n,N\in\N$. Let us explain this in more details. There exists a sentence $\Phi_{n,N}$ build from variables, symbols of 0, 1, addition $+$, multiplication $\cdot$, logical connectives $\wedge, \vee$ and quantifiers $\forall, \exists$ such that for an arbitrary field $F$ (not necessarily finite)
\begin{quote}
   $\Phi_{n,N}$ is true in $F$ (that is $F\models \Phi_{n,N}$) if and only if (\ref{eq:ls}) holds for $H=\SL_n(F)$.
\end{quote}

In order to build such $\Phi_{n,N}$, we treat a matrix of dimension $n \times n$ as a tuple of length $n^2$. Observe that
\begin{itemize}
    \item addition and multiplication of matrices can be expressed by a first order formula;
    
    \item the fact that rank of a matrix $A\in\SL_n(F)$ is greater than $k$ can be expressed by a first order quantifier free $L$-formula $\phi_{n,k}$, saying that some minor of $A$ of dimension $k \times k$ has a non-zero determinant;
    
    \item in particular, the condition $\ll_J(A)\cdot N \geq C$ can be expressed by \[\psi_{n,N}(A) = \left(\forall \lambda\neq 0\right) \phi_{n,\frac{C\cdot n}{N}}(A-\lambda I_n);\]
    
    \item for fixed $n$ and $N$, the statement (\ref{eq:ls}) can be expressed as:
\[\Phi_{n,N} = \forall A, B\ \psi_{n,N}(A) \rightarrow (\exists T_1,T_2, \ldots, T_{N})\ B = {A^{\pm1}}^{T_1}{A^{\pm1}}^{T_2}\ldots {A^{\pm1}}^{T_{N}},\]
where capital letters represent tuples of length of $n^2$, which are elements of $\SL_n(F)$.
\end{itemize}

\begin{lemma}\label{lem:field}
For all $n,N\in\N_{n>2}$, $\Phi_{n,N}$ is true in $\CC$, field of complex number. 
\end{lemma}
\begin{proof}
For any prime number $p$, an algebraic closure $\overline{\F_p}$ of $\F_p$ is a union of increasing family of finite fields. Sentence $\Phi_{n,N}$ is true in any finite field by Lemma \ref{lem:cc}
\begin{claim}
$\Phi_{n,N}$ is true in $\overline{\F_p}$ for any prime number $p$.
\end{claim}
If not, there are: a prime number $p$ and matrices  $A,B \in \SL_n(\overline{\F_p})$ such that $\ll_J(A)>\frac{C}{N}$ and \[B\not\in C_N(A,\SL_n(\overline{\F_p})).\] Take $F \subset \overline{\F_p}$ a finite field which contains all the coefficients  of matrices $A$ and $B$. Observe that $\Phi_{n,N}$ is not true in $F$, as $\ll_J(A)$ computed in $F$ cannot be smaller that $\ll_J(A)$ computed in $\F_p$, contradiction with Lemma \ref{lem:cc}.

Therefore $\Phi_{n,N}$ is true in $\overline{\F_p}$, for all prime $p\in \Pp$. By standard model-theoretic argument $\Phi_{n,N}$ is also true in any algebraically closed field of positive characteristic, so $\Phi_{n,k}$ is true in any algebraically closed field of characteristic 0, so in $\CC$.
\end{proof}

\begin{theorem} \label{PslC}
Let $G_n = (\SL_{m_n}(\CC),\ll_{J})$, for some $m_n\in\N_{>1}$.
Any metric ultraproduct of $(G_n)_{n\in\N}$ is a simple group. In fact, it is metrically  uniformly  simple. (Definition \ref{msf}).
\end{theorem}

\begin{proof}
Notice that $\Phi_{n,N}$ is true in $\CC$ for any $n>2, k>0$ by Lemma \ref{lem:field}. That is, $\G=(G_n)_{n>2}$ is metrically uniformly simple (see Definition \ref{msf}), so by Theorem \ref{thm:SimpFam} and Remark \ref{MetPrdMUS}, any metric ultraproduct of $\G$ is simple, in fact  metrically  uniformly  simple.
\end{proof}

\begin{remark}
A length function $\ll_J$ is constant on cosets of subgroup $Z(\SL_n(\CC))$ for any $n>0$. We can consider $\ll_J$ as a length on $\PSL_n(\CC)$. By Lemma \ref{lem:field} a family $(\PSL_n(\CC))_{n \in \N}$  is metrically uniformly simple, and so its metric ultraproduct.
\end{remark}


\section{Direct limits} \label{sec:dir_lim}
 
Let $(I, \leq)$  be a directed set and $\G=(G_i, \|\cdot \|_i,f_{i,j})_{i\leq j \in I}$ be a \emph{direct system of metric groups}. By this we mean that for any $i\leq j \in I$ there is an isometric homomorphism $f_{i,j}\colon G_i \rightarrow G_j$ that satisfy:
\begin{quote}
\begin{enumerate}
\item $f_{i,i}$ is the identity of $G_i$,
\item $f_{i,k}=f_{j,k} \circ f_{i,j}$ for all $i\leq j\leq k$.
\end{enumerate}
\end{quote}
A \emph{direct limit} $\varinjlim\G$ of the directed system $\G=(G_i, \|\cdot \|_i,f_{i,j})_{i\leq j \in I}$ is a group defined as follows.
 Its underlying set is \[\coprod_{i\in I} G_i /\sim,\] that is, the disjoint union
$\coprod_{i\in I} G_i $ of $\{G_i\}_{i\in I}$ modulo the following equivalence relation $\sim$ defined for $g \in G_i, \ h \in G_j$ as:
\begin{quote}
$g\sim h$ if and only if there is $k>i,j$ such that $f_{i,k}(g)=f_{j,k}(h)$.
\end{quote}
A group operation in $\varinjlim\G$ is defined as usual: for $g \in G_i, \ h \in G_j$ a product is given by formula: 
\[[g]_{\sim} \cdot [h]_{\sim}= \left[f_{i,k}(g) \cdot f_{j,k}(h)\right]_{\sim}, \]
where $k>i,j$.

We can also define on $\varinjlim\G$ a natural pseudo-norm:
\[\left\|[g]_{\sim}\right\| = \|g\|_i,\text{ if }g\in G_i.\]

\begin{theorem} \label{DirLim}
Let $\G=(G_i, \|\cdot \|_i,f_{i,j})_{i\leq j \in I}$ be a direct system of metric groups. Suppose that for any reals $0<r<t$ there is $N \in \N$ such that for any $i,j \in I$ and $g \in G_i$, $h \in G_j$ with $\|g\|>r, \|h \|<t$
there is $k \in I$, $k>i,j$ such that: 
\[f_{j,k}(h) \in C_N(f_{i,k}(g), G_k). \]
Then the direct limit $G=\varinjlim \G$ is metrically uniformly simple, so its metric ultrapower $G^*_{\text{met, fin}}$ is a metrically uniformly simple group (see Definition \ref{msf} (1)).
\end{theorem}

\begin{proof}
First, let us show that $G=\varinjlim \G$ is metrically uniformly simple. 
Take numbers $0<r<t$ and a natural number $N$ good for $r,t$. Let 
$[g]_{\sim},
 [h]_{\sim} \in G$ be such that $\left\|[g]_{\sim}\right\|>r$ and $\left\|[h]_{\sim}\right\|<t$. It is enough to show that
 \[h_k \in C_N(g_k, G_k)\]
 for some $k \in I$ and $g_k, h_k \in G_k$ such that $g_k \sim g$ and $h_k \sim h$.
Suppose that  $g \in G_i$ and $h \in G_j$.
Take $k \in I$ such that $k>i,j$ and let 
$g_k=f_{i,k}(g)$, $h_k=f_{j,k}(h)$.
The assumption gives that 
$f_{j,k}(h) \in C_N(f_{i,k}(g), G_k)$ so,
$h_k \in C_N(g_k, G_k)$.
So, a direct limit $G=\varinjlim \G$ is uniformly metrically simple, finally by  Remark \ref{MetPrdMUS} its metric ultrapower $G^*_{\text{met, fin}}$ is metrically uniformly simple too.
\end{proof}

\begin{example}\label{exampel}
Let us apply Theorem \ref{DirLim} to $\mathcal{SL}=(\SL_n(F),\ll_J)_{n\in\N}$, where $F$ be a finite field or the field of complex numbers and $\ll_J$ is the Jordan length (\ref{jordan}). A as result we obtain a metrically uniformly simple group $\varinjlim\mathcal{SL}$.

Consider a directed set $I=(\N_{>0}, |)$ of positive natural numbers, where $|$ is the dividing relation. For $n,m\in\N_{>0}$, with $n|m$, let 
\[f_{n,m}\colon \SL_n(F) \rightarrow \SL_m(F)\]
be defined as follows: for a matrix $A \in F^{n \times n}$,   $f_{n,m}(A)$ is a matrix of dimension $m \times m$ which has $\frac{m}{n}$ copies of $A$ along the diagonal, that is:
\[f_{n,m}(A)=\left[\begin{array}{cccc}
A &0&\ldots&0\\
0& A &\ldots&0\\
\vdots & \vdots &\ddots& \vdots \\
0&0&\ldots& A
\end{array}\right].\]
 
For any $n,m$ such that $n|m$, a function $f_{n,m}$ is an isometrical homomorphism of groups. 
For any reals $0<r<t$ take $N>C \cdot r$ where  constant $C$ is as in Lemma \ref{lem:cc}. For  $g \in G_i$, $h \in G_j$ with $\ll_J(g)>r, \ll_J(h)<t$  take the least common multiple $k=\lcm(i,j)$. We have that $\ll_J(f_{i,k}(g))>r$ and $\ll_J(f_{j,k}(h))<t$.

By Lemma \ref{lem:cc} (when $F$ is finite) and Theorem \ref{PslC} (when $F= \CC$) 
we have that $f_{i,k}(g) \in C_N(f_{j,k}(h),G_k)$.
Theorem \ref{DirLim} applied to $\mathcal{SL}$ gives that  $\varinjlim \mathcal{SL}$ is uniformly metrically simple.
\end{example}

\bibliography{sofic}
\bibliographystyle{alpha}

\end{document}